\documentclass[oneside,english]{amsart}
\usepackage[T1]{fontenc}
\usepackage[utf8]{inputenc}
\usepackage{amsthm}
\usepackage{amssymb}
\usepackage{esint}
\usepackage[all]{xy}
\usepackage{lmodern}
\usepackage{color}

\makeatletter

\numberwithin{equation}{section}
\numberwithin{figure}{section}
\theoremstyle{plain}
\newtheorem{thm}{\protect\theoremname}[section]
  \theoremstyle{plain}
  \newtheorem{prop}[thm]{\protect\propositionname}
  \theoremstyle{definition}
  \newtheorem{defn}[thm]{\protect\definitionname}
  \theoremstyle{remark}
  \newtheorem{rem}[thm]{\protect\remarkname}
  \theoremstyle{plain}
  \newtheorem{lem}[thm]{\protect\lemmaname}
  \theoremstyle{definition}
  
  \newtheorem{cor}[thm]{\protect\corollaryname}
  \theoremstyle{plain}

\makeatother

\usepackage{babel}
  \providecommand{\definitionname}{Definition}
  \providecommand{\examplename}{Example}
  \providecommand{\lemmaname}{Lemma}
  \providecommand{\propositionname}{Proposition}
  \providecommand{\remarkname}{Remark}
  \providecommand{\corollaryname}{Corollary}
\providecommand{\theoremname}{Theorem}

\begin{document}

\title{Maurer-Cartan spaces of filtered $L_{\infty}$-algebras}

\address{University of Luxembourg, Mathematics Research Unit,
6 rue Richard Coudenhove-Kalergi L-1359 Luxembourg}

\email{sinan.yalin@uni.lu}

\author{Sinan Yalin}
\begin{abstract}
We study several homotopical and geometric properties of Maurer-Cartan spaces for $L_{\infty}$-algebras
which are not nilpotent, but only filtered in a suitable way. Such algebras play a key role especially in
the deformation theory of algebraic structures. In particular, we prove that the Maurer-Cartan simplicial set preserves fibrations
and quasi-isomorphisms. Then we present an algebraic geometry viewpoint on Maurer-Cartan moduli sets, and we compute the
tangent complex of the associated algebraic stack.
\end{abstract}
\maketitle

\tableofcontents{}

\section*{Introduction}

A guiding principle of deformation theory is to consider that any moduli problem can be interpreted as
the deformation functor of a certain dg Lie algebra, called a deformation complex. That is,
deformation problems are controlled by dg Lie algebras.
This deformation functor is actually built from the Maurer-Cartan elements of this algebra, up to
an action of an algebraic group called the gauge group. This viewpoint emerged as a convergence of ideas
coming from Grothendieck, Schlessinger, Deligne and Drinfeld among others.
Such problems originally arose in geometry, for instance in the pioneering work of Kodaira-Spencer about deformation theory of complex manifolds \cite{KS}
(see also \cite{Man2} for a modern account of the topic).
Deformation problems of algebraic structures have been also adressed through this approach, starting with the work of Gerstenhaber \cite{Ger},
and generalized to algebras over operads (see Chapter 12 of \cite{LV}) and properads (see \cite{MV}).
A rigorous formulation of this guiding principle has been obtained in \cite{Lur} as an equivalence of $(\infty,1)$-categories between formal moduli problems and dg Lie algebras.

This article presents various aspects of Lie theory and deformation theory for $L_{\infty}$-algebras which are not nilpotent, but only filtered
in a suitable way. For this, we study several homotopical and geometric properties of their Maurer-Cartan simplicial sets and Maurer-Cartan moduli sets.
The main motivation comes from the deformation theory of algebraic structures, which relies on differential graded Lie algebras and $L_{\infty}$ algebras of this kind.
For this we refer the reader to Chapter 12 of \cite{LV}, or the papers \cite{Mar} and \cite{MV} in a more general setting.
A well known example of a deformation complex which is not a nilpotent Lie algebra but just a filtered $L_{\infty}$-algebra
is the Gerstenhaber-Schack complex, encoding deformations of bialgebras (see \cite{GS} and \cite{SS}).

The first main result extends Getzler's invariance results \cite{Get} to filtered
$L_{\infty}$ algebras. Let us note $\{F_r\mathfrak{g}\}_{r\in\mathbb{N}}$ the filtration of a filtered $L_{\infty}$ algebra $\mathfrak{g}$ and $\hat{\mathfrak{g}}$ the completion
of $\mathfrak{g}$ with respect to this filtration. Let $MC_{\bullet}(\hat{\mathfrak{g}})$ be the simplicial Maurer-Cartan set of $\hat{\mathfrak{g}}$ and $\mathcal{MC}(\hat{\mathfrak{g}})=\pi_0MC_{\bullet}(\hat{\mathfrak{g}})$ the
Maurer-Cartan moduli set of $\hat{\mathfrak{g}}$. Our result reads:
\begin{thm}
(1) Let $f:\mathfrak{g}\twoheadrightarrow \mathfrak{h}$ be a surjection of filtered $L_{\infty}$ algebras inducing surjections $F_rf:F_r\mathfrak{g}\twoheadrightarrow F_r\mathfrak{h}$ for every integer $r$.
Suppose that for every integer $r$, the $L_{\infty}$ algebras $\mathfrak{}g/F_r\mathfrak{g}$ and $\mathfrak{h}/F_r\mathfrak{h}$ are profinite.
Then $MC_{\bullet}(\hat{f})$ is a fibration of Kan complexes.

(2) Let $f:\mathfrak{g}\stackrel{\sim}{\rightarrow} \mathfrak{h}$ be a quasi-isomorphism of filtered $L_{\infty}$ algebras inducing quasi-isomorphisms $F_rf:F_r\mathfrak{g}\stackrel{\sim}{\rightarrow} F_r\mathfrak{h}$
for every integer $r$.
Suppose that $\mathfrak{g}$ and $\mathfrak{h}$ are concentrated in non-negative degrees.
Then $MC_{\bullet}(\hat{f})$ is a weak equivalence of Kan complexes.
\end{thm}
The proof relies on homotopical algebra arguments and the results of \cite{Get}, using the fact that under the aforementioned assumptions,
the Maurer-Cartan simplicial set is the limit of a tower of fibrations. Let us note that Part (2) of Theorem 0.1 has an analogue for $L_{\infty}$-quasi-isomorphisms
of $L_{\infty}$-algebras which is the subject of \cite{DR}.
As a corollary we get:
\begin{cor}
(1) A quasi-isomorphism of filtered $L_{\infty}$-algebras concentrated in non-negative degrees $f:\mathfrak{g}\rightarrow \mathfrak{h}$, restricting to a quasi-isomorphism
at each stage of the filtration induces a bijection of Maurer-Cartan moduli sets
\[
\mathcal{MC}(\hat{\mathfrak{g}})\cong \mathcal{MC}(\hat{\mathfrak{h}}).
\]

(1) A quasi-isomorphism of filtered dg Lie algebras concentrated in non-negative degrees $f:\mathfrak{g}\rightarrow \mathfrak{h}$, restricting to a quasi-isomorphism
at each stage of the filtration induces a bijection of Maurer-Cartan moduli sets
\[
MC(\hat{\mathfrak{g}})/exp(\hat{\mathfrak{g}}^0)\cong MC(\hat{\mathfrak{g}})/exp(\hat{\mathfrak{h}}^0).
\]
\end{cor}
The first part of Corollary 0.2 follows from Theorem 0.1 by taking the connected components of the simplicial Maurer-Cartan sets.
The second part follows from the first part via an identification between the connected components of the simplicial Maurer-Cartan set and the quotient of Maurer-Cartan elements by the action of the gauge group. This bijection was previously known for nilpotent dg Lie algebras \cite{Man}.
The generalization of this bijection to complete dg Lie algebras relies in particular on the technical results of Appendix C in \cite{Dol}.
Another corollary of Theorem 0.1 gives the following homotopical properties of the simplicial Maurer-Cartan functor:
\begin{cor}
Let $\mathfrak{g}$ be a filtered $L_{\infty}$ algebra.
The functor
\[
MC_{\bullet}(\mathfrak{g}\hat{\otimes}-):CDGA_{\mathbb{K}}\rightarrow sSet
\]
preserves fibrations and weak equivalences of cdgas.
\end{cor}

We are then interested in an interpretation of Maurer-Cartan moduli sets and the associated deformation functors
in algebraic geometry.
Let $\mathfrak{g}$ be a filtered dg Lie algebra. On the one hand, there is a notion of global tangent space for the associated deformation
functor $Def_{\mathfrak{g}}=MC(\mathfrak{g}\hat{\otimes}-)/exp(\mathfrak{g}^0\hat{\otimes}-)$, which is simply defined by evaluating
this functor on the algebra of dual numbers $\mathbb{K}[t]/(t^2)$ (see \cite{Sch}). This vector
space $t_{Def_{\mathfrak{g}}}$ is the first cohomology group of $\hat{\mathfrak{g}}$.
On the other hand, under some suitable assumptions, Maurer-Cartan elements $MC(\hat{\mathfrak{g}})$ form a scheme with an action
of the prounipotent algebraic group $exp(\hat{\mathfrak{g}})$. The associated quotient stack $[MC(\hat{\mathfrak{g}})/exp(\hat{\mathfrak{g}})]$
admits a notion of tangent complex over a given point, that is, over a Maurer-Cartan element of $\hat{\mathfrak{g}}$.
The precise link between these two notions does not seem to appear in the literature, so we fill the gap and prove:
\begin{thm}
Let $g$ be a filtered dg Lie algebra.
Suppose that for every integer $r$, the vector space of degree $1$ elements of $\mathfrak{g}/F_r\mathfrak{g}$ is of finite dimension.
Let $\varphi$ be a Maurer-Cartan element of $\hat{\mathfrak{g}}$ and $[\varphi]\in\mathcal{MC}(\hat{\mathfrak{g}})$ its equivalence class.
The $0^{th}$ cohomology group of the tangent complex of the algebraic stack $[MC(\hat{\mathfrak{g}})/exp(\hat{\mathfrak{g}}^0)]$ at $[\varphi]$ is the global tangent space of the deformation
functor $Def_{\hat{\mathfrak{g}}^{\varphi}}$, that is,
\[
H^0\mathbb{T}_{[MC(\hat{\mathfrak{g}})/exp(\hat{\mathfrak{g}}^0)],[\varphi]} = t_{Def_{\hat{\mathfrak{g}}^{\varphi}}} = H^1\hat{\mathfrak{g}}^{\varphi}.
\]
\end{thm}
This applies in particular to the deformation theory of algebraic structures. Indeed, if $P$ is a
Koszul dg properad and $X$ a chain complex, the corresponding deformation complex is defined by $\mathfrak{g}_{P,X}=Hom_{\Sigma}(\overline{C},End_X)$.
It is a complete $L_{\infty}$ algebra whose Maurer-Cartan elements are the $P_{\infty}$-algebra structures on $X$.
Under a certain technical condition about the Koszul dual of $P$, the complex $\mathfrak{g}_{P,X}$
satisfies the required assumption of Theorem 0.3. This works for a broad range of algebras and bialgebras:
\begin{itemize}
\item Associative algebras, commutative algebras, Lie algebras, Poisson algebras;
\item Associative and coassociative bialgebras;
\item Lie bialgebras (mathematical physics, quantum group theory);
\item Involutive Lie bialgebras (string topology);
\item Frobenius bialgebras (Poincaré duality, field theories).
\end{itemize}

\noindent
\textit{Organization of the article.}
The first section recalls the two equivalent definitions of $L_{\infty}$ algebras, the notion
of filtration and the extension of a filtered $L_{\infty}$ algebra by a commutative differential graded algebra.
Section 2 starts with the definition of Maurer-Cartan elements and the important operation of twisting by
a Maurer-Cartan element. In 2.2 we define the simplicial Maurer-Cartan set and the Maurer-Cartan moduli set.
In 2.3 we deal with filtered dg Lie algebras and show that in this setting, there is a bijection between
connected components of the simplicial Maurer-Cartan set and equivalences classes of Maurer-Cartan elements
under the action of the gauge group. Section 3 is devoted to the proofs of Theorem 0.1 and Theorem 0.2.
In 3.3 we conclude with an invariance result under the twisting operation. The algebraic geometry intepretation
is the subject of Section 4. In 4.1 we recall the notion of an action groupoid in a category, and the definition
of the stack associated to the action of a smooth affine group scheme on a noetherian scheme. In 4.2 and
4.3 we prove Theorem 0.3. We start with some recollections about the tangent complex of an algebraic stack, apply
this notion to the quotient of a scheme by a smooth algebraic group, then explicit the tangent complex
in the case of the quotient of the scheme of Maurer-Cartan elements by the action of the gauge group.
In 4.4 we explain briefly how to weaken the assumption that the vector space of degree 1 elements in the Lie algebra is of finite dimension.
In 4.5 we show how it applies to algebraic deformation theory in the formalism of algebras over
Koszul properads. We conclude with a remark about a possible generalization of this geometric interpretation
to $L_{\infty}$ deformation complexes.

\noindent
\textit{Acknowledgements.} I would like to thank Gennaro di Brino and Bertrand Toen for useful discussions about stacks.

\section{Filtered $L_{\infty}$ algebras}

Throughout this paper, we work in the category $Ch_{\mathbb{K}}$ of unbounded chain complexes over a field $\mathbb{K}$ of
characteristic zero. We adopt the cohomological grading convention,  so the differentials raise degrees by $1$.
For any chain complex $X$, we denote by $sX$ its suspension, defined by $sX^n=X^{n-1}$ for every integer $n$.
We say that a chain complex is finite if it is bounded and of finite dimension in each degree.
A profinite chain complex is a sequential limit of finite chain complexes.

\subsection{$L_{\infty}$ algebras}

There are two equivalent ways to define $L_{\infty}$ algebras:
\begin{defn}
(1) An $L_{\infty}$ algebra is a graded vector space $\mathfrak{g}=\{\mathfrak{g}_n\}_{n\in\mathbb{Z}}$ equipped with maps
$l_k:\mathfrak{g}^{\otimes k}\rightarrow \mathfrak{g}$ of degree $k-2$, for $k\geq 1$, satisfying the following properties:
\begin{itemize}
\item $[...,x_i,x_{i+1},...]=-(-1)^{|x_i||x_{i+1}|}[...,x_{i+1},x_i,...]$
\item for every $k\geq 1$, the generalized Jacobi identities
\[
\sum_{i=1}^k\sum_{\sigma\in Sh(i,k-i)}(-1)^{\epsilon(i)}[[x_{\sigma(1)},...,x_{\sigma(i)}],x_{\sigma(i+1)},...,x_{\sigma(k)}]=0
\]
where $\sigma$ ranges over the $(i,k-i)$-shuffles and
\[
\epsilon(i) = i+\sum_{j_1<j_2,\sigma(j_1)>\sigma(j_2)}(|x_{j_1}||x_{j_2}|+1).
\]
\end{itemize}

(2) An $L_{\infty}$ algebra structure on a graded vector space $\mathfrak{g}=\{\mathfrak{g}_n\}_{n\in\mathbb{Z}}$ is a
coderivation of cofree coalgebras $Q:\Lambda^{\bullet\geq 1}s\mathfrak{g}\rightarrow \Lambda^{\bullet\geq 1}s\mathfrak{g}$ of degree $-1$ such that $Q^2=0$,
where the notation $\Lambda^{\bullet\geq 1}s\mathfrak{g}$ stands for the exterior powers on $s\mathfrak{g}$ without the weight zero term
$\Lambda^0s\mathfrak{g}=\mathbb{K}$.
\end{defn}
The generalized Jacobi relations imply in particular that the bracket $l_1$ is a differential, so we can speak about differential graded (dg for short) $L_{\infty}$
algebras. One recover the first definition from the second via the following argument.
The coderivation $Q$ is, by definition, entirely determined by its values on the generators, that is by the composite
\[
Q_{proj}:\Lambda^{\bullet\geq 1}s\mathfrak{g}\stackrel{Q}{\rightarrow}\Lambda^{\bullet\geq 1}s\mathfrak{g}\stackrel{proj}{\rightarrow}s\mathfrak{g}.
\]
This map is equivalent to the datum of a collection of maps
\[
Q^k:\Lambda^ks\mathfrak{g}\hookrightarrow\Lambda^{\bullet\geq 1}s\mathfrak{g}\stackrel{Q_{proj}}{\rightarrow} s\mathfrak{g}.
\]
These maps, in turn, determine the brackets $l_k$ via the formulae
\[
l_k=s^{-1}\circ Q^k\circ s^{\otimes k}:\Lambda^k\mathfrak{g}\rightarrow\Lambda^ks\mathfrak{g}\rightarrow s\mathfrak{g}\rightarrow \mathfrak{g}.
\]
In the converse direction, one builds the coderivation $Q$ from the brackets $l_k$.
The dual of the quasi-cofree coalgebra $(\Lambda s\mathfrak{g},Q)$ obtained in this way is called the Chevalley-Eilenberg algebra of $\mathfrak{g}$
and will be noted $C^*(\mathfrak{g})$.
\begin{rem}
If $l_k=0$ for $k\geq 2$ then $\mathfrak{g}$ is a chain complex. If $l_k=0$ for $k\geq 3$ then $\mathfrak{g}$ is a dg Lie algebra.
\end{rem}

\subsection{Filtrations}

We use Definition 2.1 of \cite{Ber}.
An $L_{\infty}$ algebra $\mathfrak{g}$ is filtered if it admits a decreasing filtration
\[
\mathfrak{g}=F_1\mathfrak{g}\supseteq F_2\mathfrak{g}\supseteq...\supseteq F_r\mathfrak{g}\supseteq ...
\]
compatible with the brackets: for every $l\geq 1$,
\[
[F_r\mathfrak{g},\mathfrak{g},...,\mathfrak{g}]\subset F_{r+1}\mathfrak{g}.
\]
This compatibility implies that the quotients $\mathfrak{g}/F_r\mathfrak{g}$ are still $L_{\infty}$ algebras.
We suppose moreover that for every $r$, there exists an integer $N(r)$ such that $[\mathfrak{g}^{\wedge l}]\subseteq F_r\mathfrak{g}$
for every $l>N(r)$, where the short notation $[\mathfrak{g}^{\wedge l}]$ stands for the $l^{th}$ bracket applied to any $l$-tuple of elements of $\mathfrak{g}$.
Then the quotients $\mathfrak{g}/F_r\mathfrak{g}$ are nilpotent $L_{\infty}$-algebras in the sense of \cite{Get}, Definition 4.2.

We denote by $dgL_{\infty}^{filt}$ the category of filtered $L_{\infty}$ algebras, with morphisms the morphisms
of $L_{\infty}$ algebras $f:\mathfrak{g}\rightarrow \mathfrak{h}$ inducing morphisms $F_rf:F_r\mathfrak{g}\rightarrow F_r\mathfrak{h}$ for every $r$.

\subsection{Extension by artinian algebras and cgdas}

We denote by $Art_{\mathbb{K}}$ the category of local artinian unitary commutative $\mathbb{K}$-algebras of residue field $\mathbb{K}$. These are finitely generated unitary commutative
$\mathbb{K}$-algebras $A$ with a unique maximal ideal $m_A$, so that $A/m_A\cong\mathbb{K}$ (hence a decomposition $A=m_A\oplus\mathbb{K}$ as vector spaces).
Moreover, such an algebra $A$ satisfies the descending chain condition on ideals: every totally ordered subset of the set of ideals of $A$ (with the partial order given by the inclusion)
admits a minimum.
For every $A\in Art_{\mathbb{K}}$ with maximal ideal $m_A$, there is a canonical $L_{\infty}$-structure on $\mathfrak{g}\otimes m_A$ extending the one of $\mathfrak{g}$.
For every $\{\tau_i\otimes a_i\in \mathfrak{g}\otimes m_A\}_{1\leq i\leq k}$, the brackets of this structure are defined by
\[
l_k^A(\tau_1\otimes a_1,...,\tau_k\otimes a_k)= l_k(\tau_1,...,\tau_k)\otimes a_1.a_2...a_k.
\]
The filtration of $\mathfrak{g}\otimes m_A$ trivially extends the one of $\mathfrak{g}$, that is
\[
F_r(\mathfrak{g}\otimes m_A)=F_r\mathfrak{g}\otimes m_A.
\]

Let us note $CDGA_{\mathbb{K}}$ the category of commutative differential graded algebras.
Let $A$ be a cdga, then $\mathfrak{g}\otimes A$ is a filtered $L_{\infty}$ algebra by the same construction as above.
Both constructions are functorial in $A$, hence functors
\[
\mathfrak{g}\otimes - :Art_{\mathbb{K}}\rightarrow dgL_{\infty}^{filt}
\]
and
\[
\mathfrak{g}\otimes - :CDGA_{\mathbb{K}}\rightarrow dgL_{\infty}^{filt}.
\]

\section{Completion and Maurer-Cartan spaces}

We define the Maurer-Cartan elements, the twisting operation, the simplicial Maurer-Cartan set and the moduli set of Maurer-Cartan elements.
In the case of dg Lie algebras, we compare this moduli set with the one obtained from the gauge group action.

\subsection{Completion and the Maurer-Cartan series}

\subsubsection{Convergent series}

In a dg Lie algebra, Maurer-Cartan elements are degree $1$ solutions of the Maurer-Cartan equation
\[
d(x)+\frac{1}{2}[x,x]=0,
\]
where $d$ is the differential of the dg Lie algebra and $[-,-]$ its Lie bracket.
One would like to have well defined Maurer-Cartan elements in an $L_{\infty}$-algebra.
In this situation, we have to define an equation involving all the brackets $l_k$.
Let $\mathfrak{g}$ be a dg $L_{\infty}$-algebra and $\tau\in \mathfrak{g}^1$ an element of degree $1$, we set
\[
\mathcal{F}(\tau)=\sum_{k\geq 1} \frac{1}{k!} [\tau^{\wedge k}]
\]
and say that $\tau\in MC(\mathfrak{g})$ if $\mathcal{F}(\tau)=0$.
But $\mathcal{F}(\tau)$ is a priori a divergent series, so we need the following local nilpotence condition:
\[
\forall \tau\in \mathfrak{g}, \exists K / [\tau^{\wedge k}]=0 \, \text{for every $k>K$},
\]
which implies that for every $\tau$ the expression $\mathcal{F}(\tau)$ is polynomial in $\tau$.
Assuming that $\mathfrak{g}$ is nilpotent in the sense of Definition 4.2 in \cite{Get} (i.e its lower central series terminates) implies that it is locally nilpotent.

To be filtered is not sufficient to get convergent series in $\mathfrak{g}$, but is enough to get convergent series
in the completion $\hat{\mathfrak{g}}=lim \mathfrak{g}/F_r\mathfrak{g}$ of $\mathfrak{g}$ with respect to this filtration. This is a complete
$L_{\infty}$-algebra in the sense of \cite{Ber}.
We thus define Maurer-Cartan elements of $\mathfrak{g}$ as Maurer-Cartan elements of $\hat{\mathfrak{g}}$.
We denote by $MC(\hat{\mathfrak{g}})$ the set of Maurer-Cartan elements.

\subsubsection{Twisting by a Maurer-Cartan element}

An important operation on a complete $L_{\infty}$ algebra is the twisting by a Maurer-Cartan element:
\begin{defn}
Let $\mathfrak{g}$ be a complete $L_{\infty}$ and $\varphi$ be a Maurer-Cartan element of $\mathfrak{g}$.
The twisting of $\mathfrak{g}$ by $\varphi$, or twisted $L_{\infty}$-algebra, is the $L_{\infty}$ algebra $\mathfrak{g}^{\varphi}$
with the same underlying vector space and new brackets $l_k^{\varphi}$ defined by
\[
l_k^{\varphi}(x_1,...,x_k)=\sum_{i\geq 0}\frac{1}{i!}l_{i+k}(\varphi^{\wedge i},x_1,...,x_k).
\]
\end{defn}
The following lemma shows that the twisting operation preserves the properties of the filtration:
\begin{lem}
Let $\mathfrak{g}$ be a complete $L_{\infty}$ algebra and $\varphi$ be a Maurer-Cartan element of $\mathfrak{g}$. Then $\mathfrak{g}^{\varphi}$ is a complete $L_{\infty}$ algebra for the
same filtration of the underlying chain complex.
\end{lem}
\begin{proof}
Recall that the brackets of $\mathfrak{g}^{\varphi}$ are given by
\[
l_k^{\varphi}(x_1,...,x_k)=\sum_{i\geq 0}\frac{1}{i!}l_{i+k}(\varphi^{\wedge i},x_1,...,x_k).
\]
First, the filtration is compatible with the new differential $d_{\varphi}$ of $\mathfrak{g}^{\varphi}$. Indeed, this differential is defined by
\begin{eqnarray*}
d_{\varphi}(x) & = & \sum_{k\geq 0}\frac{1}{k!}[\varphi^{\wedge k},x] \\
  & = & d_{\mathfrak{g}}(x) + \sum_{k\geq 1}\frac{1}{k!}[\varphi^{\wedge k},x]
\end{eqnarray*}
where $d_{\mathfrak{g}}$ is the differential of $\mathfrak{g}$. If $x\in F_r\mathfrak{g}$, then $d_{\mathfrak{g}}(x)\in F_r\mathfrak{g}$ because $d_{\mathfrak{g}}$ is compatible with the filtration of $\mathfrak{g}$,
and for every integer $k$ the term $[\varphi^{\wedge k},x]$ lies in $F_{r+1}$ because $\mathfrak{g}$ is complete, hence $d_{\varphi}(x)\in F_r\mathfrak{g}$.

Secondly, for any $x_1,...,x_k\in \mathfrak{g}$, if one of the $x_j$, $1\leq j\leq k$, lies in $F_r\mathfrak{g}$ for a certain integer $r$ then
for every integer $i$ the term $l_{i+k}(\varphi^{\wedge i},x_1,...,x_k)$ lies in $F_{r+1}\mathfrak{g}$ by completeness of $\mathfrak{g}$ so
$l_k^{\varphi}(x_1,...,x_k)\in F_{r+1}\mathfrak{g}$.

Thirdly, for every integer $r$ there exists an integer $N(r)$ such that $l_n{\varphi}(x_1,...,x_n)\in F_r\mathfrak{g}$ for every $n>N(r)$, so
for $k>N(r)$ every term $l_{i+k}(\varphi^{\wedge i},x_1,...,x_k)$ lies in $F_r\mathfrak{g}$, hence $l_k^{\varphi}(x_1,...,x_k)\in F_r\mathfrak{g}$.

Finally we obviously still have an isomorphism $\mathfrak{g}^{\varphi}\cong lim_r \mathfrak{g}^{\varphi}/F_r\mathfrak{g}$ where $F_r\mathfrak{g}$ is equipped with the differential $d_{\varphi}$.

\end{proof}

\subsubsection{Extension by algebras}

To consider Maurer-Cartan elements of extensions of $\mathfrak{g}$ by artinian algebras or cdgas $\mathfrak{g}\otimes A$,
we need to complete this tensor product with respect to the filtration $F_r(\mathfrak{g}\otimes A)=F_r\mathfrak{g}\otimes A$.
We denote by $\mathfrak{g}\hat{\otimes}A$ this completed tensor product.
The completion process is also functorial in $A$, so we obtain functors
\[
MC(\mathfrak{g}\hat{\otimes} -) :Art_{\mathbb{K}}\rightarrow Set
\]
and
\[
MC(\mathfrak{g}\hat{\otimes} -) :CDGA_{\mathbb{K}}\rightarrow Set.
\]

In the case where $\mathfrak{g}$ is complete, the functor of artinian algebras does not need the completed tensor product by the following arguments:
\begin{lem}
Let $A$ be a local artinian commutative $\mathbb{K}$-algebra with maximal ideal $m_A$.
There exists an integer $n(A)\in\mathbb{N}$ such that $m_A^{n(A)}=0$.
\end{lem}
\begin{proof}
The algebra $A$ is local, so it admits a unique maximal ideal $m_A$. Let us consider the totally ordered subset of ideals of $A$ defined by
\[
...\subseteq m_A^k\subseteq ...\subseteq m_A^2\subseteq m_A.
\]
The algebra $A$ is artinian, so it satisfies the descending chain condition on ideals. It means that the sequence above is stationary for a certain integer, that is,
there exists an integer $n(A)$ such that $m_A^n=m_A^{n(A)}$ for every $n>n(A)$. In particular, we have
\[
m_A.m_A^{n(A)}=m_A^{n(A)}.
\]
The Jacobson radical of $A$, noted $J(A)$, is the intersection of all maximal ideals of $A$. Since $A$ is local, we have $J(A)=m_A$.
Moreover, the algebra $A$ is finitely generated, so $m_A^{n(A)}$ is a finitely generated $A$-module. By Nakayama's Lemma (Proposition 2.6 in \cite{AM}), the equality $J(A).m_A^{n(A)}=m_A^{n(A)}$ implies that $m_A^{n(A)}=0$.
\end{proof}
We deduce:
\begin{lem}
Let $\mathfrak{g}$ be a complete $L_{\infty}$ algebra and $A$ be a local artinian commutative $\mathbb{K}$-algebra with maximal ideal $m_A$. Then $\mathfrak{g}\hat{\otimes}m_A=\mathfrak{g}\otimes m_A$.
\end{lem}
\begin{proof}
There exists an integer $k$ such that $m_A=(u_1,...,u_k)\subset A$ since $A$ is finitely generated.
Moreover, according to Lemma 2.3, there exists an integer $n$ such that $u_1^n=...=u_k^n=0$.
The space $\mathfrak{g}\hat{\otimes}m_A$ consists of sequences of polynomials $\{p_r\}_{r\in\mathbb{N}}$ of $k$ variables $u_1,...,u_k$ such that $p_r\in F_{r-1}\mathfrak{g}[u_1,...,u_k]$
and $deg(p_r)\leq deg(p_{r+1})$. This is equivalent to consider polynomials in $lim_r\mathfrak{g}/F_r\mathfrak{g}[u_1,...,u_k]$ of degree $\leq n$, which form
exactly $\mathfrak{g}\otimes m_A$ since $\mathfrak{g}$ is complete.
\end{proof}
\begin{cor}
Let $\mathfrak{g}$ be a complete $L_{\infty}$ algebra and $A$ be a local artinian commutative $\mathbb{K}$-algebra. Then $\mathfrak{g}\hat{\otimes}A=\mathfrak{g}\otimes A$.
\end{cor}

\subsection{The simplicial Maurer-Cartan space}

We have a functor
\[
MC_{\bullet}:dgL_{\infty}^{filt}\rightarrow sSet
\]
from filtered $L_{\infty}$-algebras to
simplicial sets defined by $MC_{\bullet}(\mathfrak{g})=MC(\mathfrak{g}\hat{\otimes}\Omega_{\bullet})$, where $\Omega_{\bullet}$ is
the Sullivan cdga of de Rham polynomial forms on the standard simplex $\Delta^{\bullet}$ (see \cite{Sul}).
This simplicial cdga is defined in every simplicial dimension $n$ by
\[
\Omega_n=\mathbb{K}[t_0,...,t_n,dt_0,...,dt_n]/(\sum t_i -1, \sum dt_i)
\]
where the $t_i$ are of degree $0$, the $dt_i$ are of degree $1$, and the differential sends each $t_i$ to $dt_i$ and each $dt_i$ to $0$.
The simplicial set $MC_{\bullet}(\mathfrak{g})$ is called the simplicial Maurer-Cartan set of $\mathfrak{g}$.
It is a limit of Maurer-Cartan spaces of nilpotent $L_{\infty}$-algebras:
\begin{eqnarray*}
MC_{\bullet}(\mathfrak{g}) & := & MC_{\bullet}(\hat{\mathfrak{g}})\\
 & = & lim_r MC(\hat{\mathfrak{g}}/F_r\hat{\mathfrak{g}}\hat{\otimes}\Omega^{\bullet}) \\
 & = & lim_r MC_{\bullet}(\hat{\mathfrak{g}}/F_r\hat{\mathfrak{g}})\\
 & = & lim_r MC_{\bullet}(\mathfrak{g}/F_r\mathfrak{g})\\
\end{eqnarray*}

The set $MC(\mathfrak{g}\hat{\otimes} \Omega_1)$ gives the homotopies between the Maurer-Cartan elements of $\mathfrak{g}$, and we define by
\[
\mathcal{MC}(\mathfrak{g})=\pi_0 MC_{\bullet}(\mathfrak{g})
\]
the moduli set of Maurer-Cartan elements of $\mathfrak{g}$.

\begin{prop}
For every filtered $L_{\infty}$-algebra $\mathfrak{g}$, the simplicial set $MC_{\bullet}(\hat{\mathfrak{g}})$ is a Kan complex.
\end{prop}
\begin{proof}
For every $r$, the inclusion $F_{r+1}\mathfrak{g}\subseteq F_r\mathfrak{g}$ induces a surjection of nilpotent
$L_{\infty}$-algebras $\mathfrak{g}/F_{r+1}\mathfrak{g}\twoheadrightarrow \mathfrak{g}/F_r\mathfrak{g}$, hence a fibration of Kan complexes
$MC_{\bullet}(\mathfrak{g}/F_{r+1}\mathfrak{g})\twoheadrightarrow MC_{\bullet}(\mathfrak{g}/F_r\mathfrak{g})$ according to \cite{Get}.
We thus obtain a tower of fibrations of Kan complexes
\[
MC_{\bullet}(\mathfrak{g})\rightarrow...\twoheadrightarrow MC_{\bullet}(\mathfrak{g}/F_{r+1}\mathfrak{g})\twoheadrightarrow MC_{\bullet}(\mathfrak{g}/F_r\mathfrak{g})
\twoheadrightarrow...\twoheadrightarrow pt
\]
whose transfinite composite gives a Kan fibration from $MC_{\bullet}(\mathfrak{g})$ to the point,
that is to say, $MC_{\bullet}(\mathfrak{g})$ is a Kan complex.
\end{proof}

\subsection{Homotopy versus gauge in the Lie case}

Gauge equivalence and homotopy equivalence for nilpotent dg Lie algebras have been studied in \cite{Man1},
and in \cite{ScS} for applications to the classification of rational homotopy types.
The goal of this section is to generalize the correspondence between gauge equivalences and homotopy equivalences to filtered dg Lie algebras.
For every filtered dg Lie algebra $\mathfrak{g}$ and a filtration $\mathfrak{g}=F_0\mathfrak{g}\supseteq F_1\mathfrak{g}...$,
we consider its Maurer-Cartan elements as the Maurer-Cartan elements of
its completion $\hat{\mathfrak{g}}$ (to ensure convergent series). Hence we can assume here, without loss of generality, that $\mathfrak{g}$ is complete with respect
to its filtration.
Homotopies between two Maurer-Cartan elements $x,y\in \mathfrak{g}^1$ are Maurer-Cartan elements of $\mathfrak{g}\hat{\otimes}\Omega_1$.
Let us explicit the structure of this complete dg Lie algebra. Let $\mathfrak{g}\{t\}$ be the complete dg Lie algebra of formal series $f\in \mathfrak{g}[[t]]$
of the form
\[
f(t)=\sum_{r=1}^{\infty}f_rt^r
\]
where $f_r\in F_{m_r}\mathfrak{g}$ for every integer $r$ and $\{m_r\}_{r\in\mathbb{N}}$ is an increasing sequence of integers going to $+\infty$ when $r$ goes to $+\infty$.
This Lie algebra has been introduced in Appendix C of \cite{Dol}.
\begin{prop}
We have
\[
\mathfrak{g}\hat{\otimes}\Omega_1=\mathfrak{g}\{t\}\oplus \mathfrak{g}\{t\}dt.
\]
For every $f_0+f_1dt\in \mathfrak{g}\{t\}\oplus \mathfrak{g}\{t\}dt$, the differential is given by
\[
d(f_0+f_1dt)=\delta_{\mathfrak{g}}f_0+\delta_{\mathfrak{g}}f_1dt+\frac{df_0}{dt}dt
\]
where $\delta_{\mathfrak{g}}$ is the differential of $\mathfrak{g}$.
For every $f_0+f_1dt,\tilde{f}_0+\tilde{f}_1dt\in \mathfrak{g}\{t\}\oplus \mathfrak{g}\{t\}dt$, the Lie bracket is given by
\[
[f_0+f_1dt,\tilde{f}_0+\tilde{f}_1dt]=[f_0,\tilde{f}_0]+[f_0,\tilde{f}_1]dt+(-1)^{|\tilde{f}_0||f_1|}[\tilde{f}_0,f_1]dt.
\]
\end{prop}
\begin{proof}
The algebra $\Omega_1$ is defined by $\Omega_1=\mathbb{K}[t]\oplus\mathbb{K}[t]dt$ where $t$ is a degree $0$ generator, $\mathbb{K}[t]$ is the ring of polynomials in $t$,
and $dt$ is a degree $1$ generator such that $dt^2=0$. The differential is defined by sending $t$ to $dt$ and $dt$ to $0$.
Using the canonical isomorphism $\mathfrak{g}\stackrel{\cong}{\rightarrow}lim_r\mathfrak{g}/F_r\mathfrak{g}$, we have
\[
\mathfrak{g}\hat{\otimes}\Omega_1\cong lim_r\mathfrak{g}[t]/F_r\mathfrak{g}[t]\oplus lim_r\mathfrak{g}[t]/F_r\mathfrak{g}[t] dt.
\]
The elements of $lim_r\mathfrak{g}[t]/F_r\mathfrak{g}[t]$ are sequences of polynomials $\{p_r\}_{r\in\mathbb{N}}$ where $p_r\in \mathfrak{g}/F_r\mathfrak{g}(t)$, such that the canonical projection
$\mathfrak{g}/F_r\mathfrak{g}[t]\twoheadrightarrow \mathfrak{g}/F_{r-1}\mathfrak{g}[t]$ sends $p_r$ to $p_{r-1}$. Firstly, this implies that the degree of $p_r$ increases with $r$.
Secondly, this implies that for every integer $r$, the polynomial $p_r$ can be written as $p_r=p_{r-1}+\tilde{p}_r$ where $\tilde{p_r}\in F_{r-1}\mathfrak{g}[t]$.
Applying this formula recursively, we see that elements of $lim_r\mathfrak{g}[t]/F_r\mathfrak{g}[t]$ are sequences of polynomials $\{p_r\}_{r\in\mathbb{N}}$ such that $p_r\in F_{r-1}\mathfrak{g}[t]$ and
$deg(p_r)\geq deg(p_{r-1})$. The limit of such a sequence is a formal series of the desired form, so $lim_r\mathfrak{g}[t]/F_r\mathfrak{g}[t]=\mathfrak{g}\{t\}$.

The differential of $\mathfrak{g}\hat{\otimes}\Omega_1$ is given by the componentwise differential on $\mathfrak{g}\{t\}\oplus \mathfrak{g}\{t\}dt$,
which applies the differential of $\mathfrak{g}$ to each term of the formal series and sends $t$ to $dt$. Given $f_0+f_1dt\in \mathfrak{g}\{t\}\oplus \mathfrak{g}\{t\}dt$,
the differential of $f_0$ is $\delta_{\mathfrak{g}}f_0+\frac{df_0}{dt}dt$, and the differential of $f_1dt$ is $\delta_{\mathfrak{g}}f_1dt+\frac{df_1}{dt}dt^2=\delta_{\mathfrak{g}}f_1dt$
since $dt^2=0$.

Now let $f_0+f_1dt,\tilde{f}_0+\tilde{f}_1dt\in \mathfrak{g}\{t\}\oplus \mathfrak{g}\{t\}dt$. The Lie bracket $[f_0+f_1dt,\tilde{f}_0+\tilde{f}_1dt]$ gives componentwise
\begin{eqnarray*}
[f_0+f_1dt,\tilde{f}_0+\tilde{f}_1dt] & = & [f_0,\tilde{f}_0]+[f_0,\tilde{f}_1dt]+[f_1dt,\tilde{f}_0]+[f_1dt,\tilde{f_1}dt] \\
 & = & [f_0,\tilde{f}_0]+[f_0,\tilde{f}_1]dt-[f_1,\tilde{f}_0]dt+[f_1,\tilde{f_1}]dt^2 \\
 & = & [f_0,\tilde{f}_0]+[f_0,\tilde{f}_1]dt+(-1)^{|\tilde{f}_0||f_1|}[\tilde{f}_0,f_1]dt.
\end{eqnarray*}
\end{proof}
Note that instead of the polynomials in $t$ and $dt$ of the nilpotent case (see section 5 of \cite{Man})
our homotopies are formal power series.
These homotopies define a moduli set of Maurer-Cartan elements $\mathcal{MC}(\mathfrak{g})=\pi_0MC_{\bullet}(\mathfrak{g})$.

On the other hand, the degree $0$ part $\mathfrak{g}^0$ of $\mathfrak{g}$ forms a pronilpotent Lie algebra, that is, the inverse
limit of a tower of nilpotent Lie algebras. Here this tower is given by the quotients $\mathfrak{g}^0/F_r\mathfrak{g}^0$ induced by the filtration of $\mathfrak{g}$ and the
projections $\mathfrak{g}^0/F_r\mathfrak{g}^0\twoheadrightarrow \mathfrak{g}^0/F_{r-1}\mathfrak{g}^0$.
We define a group structure on the set of elements of $\mathfrak{g}^0$ by the Baker-Campbell-Hausdorff formula
\[
x\bullet y = BCH(x,y)
\]
which is a convergent series because of the pronilpotence of $\mathfrak{g}^0$, hence exponentiates $\mathfrak{g}^0$ to a
prounipotent algebraic group $exp(\mathfrak{g}^0)$. This group is the gauge group of $\mathfrak{g}$,
and acts on the set of Maurer-Cartan elements: for every $\xi\in \mathfrak{g}^0$ and $\tau\in MC(\mathfrak{g})$,
we have
\[
exp(\xi).\tau=e^{[\xi,-]}(\tau) - \frac{e^{[\xi,-]}-id}{[\xi,-]}(\delta_{\mathfrak{g}}\xi)
\]
(see Appendix C of \cite{Dol}).
\begin{prop}
Two Maurer-Cartan elements of a complete dg Lie algebra are gauge equivalent if and only if they are homotopy equivalent.
\end{prop}
\begin{proof}
The homotopies of Maurer-Cartan elements form the set $MC(\mathfrak{g}\hat{\otimes}\Omega_1)$, that is, the set of series $f_0+f_1dt\in \mathfrak{g}\{t\}\oplus \mathfrak{g}\{t\}dt$
such that $f_0$ is of degree $1$, $f_1$ is of degree $0$ and
\[
d(f_0+f_1dt)+\frac{1}{2}[f_0+f_1dt,f_0+f_1dt]=0.
\]
According to Proposition 2.7, this gives the equation
\[
\delta_{\mathfrak{g}}f_0+\delta_{\mathfrak{g}}f_1dt+\frac{df_0}{dt}dt+\frac{1}{2}[f_0,f_0]+[f_0,f_1]dt.
\]
This is equivalent to the datum of two equations
\[
\delta_{\mathfrak{g}}f_0+\frac{1}{2}[f_0,f_0]=0
\]
and
\[
\frac{df_0}{dt}=-(\delta_{\mathfrak{g}}f_1+[f_0,f_1])=-\delta_{\mathfrak{g}}f_1+[f_1,f_0].
\]
The faces $d_0,d_1:MC(\mathfrak{g}\hat{\otimes}\Omega_1)\rightarrow MC(\mathfrak{g})$ are defined respectively by the evaluation at $t=0,dt=0$ and the evaluation at $t=1,dt=0$,
hence $d_0(f_0+f_1dt)=f_0|_{t=0}$ and $d_1(f_0+f_1dt)=f_0|_{t=1}$.
We conclude that two Maurer-Cartan elements $\tau$ and $\tau'$ are homotopic if and only if there exists two formal power series $f_0,f_1\in \mathfrak{g}\{t\}$ respectively of
degrees $1$ and $0$ such that $f_0|_{t=0}=\tau$, $f_0|_{t=1}=\tau'$ and the two equations above are satisfied.
In particular, the series $f_0$ is a solution of the differential equation
\[
\frac{df_0}{dt}=-\delta_{\mathfrak{g}}f_1++[f_1,f_0]
\]
with boundary conditions $f_0|_{t=0}=\tau$ and $f_0|_{t=1}=\tau'$. According to Theorem C.1 of \cite{Dol}, this implies that $\tau$ and $\tau'$
are connected by the action of $exp(\mathfrak{g}^0)$.

In the converse direction, suppose that $\tau$ and $\tau'$ are two Maurer-Cartan elements connected by the action of the gauge group. There exists $\xi\in \mathfrak{g}^0$
such that
\[
\tau'=e^{[\xi,-]}(\tau) - \frac{e^{[\xi,-]}-id}{[\xi,-]}(\delta_{\mathfrak{g}}\xi).
\]
We set
\[
f_0=e^{t[\xi,-]}(\tau) - \frac{e^{t[\xi,-]}-id}{[\xi,-]}(\delta_{\mathfrak{g}}\xi).
\]
This series satisfies $f_0|_{t=0}=\tau$ and $f_0|_{t=1}=\tau'$. Moreover, on the one hand,
\[
\frac{df_0}{dt} = [\xi,e^{t[\xi,-]}(\tau)]-e^{t[\xi,-]}(\delta_{\mathfrak{g}}\xi)
\]
and on the other hand
\begin{eqnarray*}
-\delta_{\mathfrak{g}}\xi+[\xi,f_0] & = & -\delta_{\mathfrak{g}}\xi + [\xi,e^{t[\xi,-]}(\tau)]-(e^{t[\xi,-]}-id)(\delta_{\mathfrak{g}}\xi) \\
 & = & [\xi,e^{t[\xi,-]}(\tau)]-e^{t[\xi,-]}(\delta_{\mathfrak{g}}\xi) \\
 & = & \frac{df_0}{dt}.
\end{eqnarray*}
Consequently, by applying Proposition C.1 of \cite{Dol}, we get a homotopy between $\tau$ and $\tau'$ defined by
\[
e^{t[\xi,-]}(\tau) - \frac{e^{t[\xi,-]}-id}{[\xi,-]}(\delta_{\mathfrak{g}}\xi) + \xi dt.
\]
\end{proof}

\section{Invariance properties}

In this section we study invariance properties of simplicial Maurer-Cartan sets. We prove Theorems 0.2 and 0.3, and conclude with an invariance result under twisting.

\subsection{Preservation of fibrations and weak equivalences}

\begin{prop}
Let $f:\mathfrak{g}\rightarrow \mathfrak{h}$ be a morphism of filtered $L_{\infty}$-algebras.

(1) If $f$ is a quasi-isomorphism inducing a quasi-isomorphism at each stage of the filtration
$F_rf:F_r\mathfrak{g}\stackrel{\sim}{\rightarrow}F_r\mathfrak{h}$, then $\hat{f}$ is a quasi-isomorphism
of complete $L_{\infty}$-algebras.

(2) If $f$ is a surjection inducing a surjection at each stage of the filtration
$F_rf:F_r\mathfrak{g}\twoheadrightarrow F_r\mathfrak{h}$, then $\hat{f}$ is a surjection
of complete $L_{\infty}$-algebras.
\end{prop}
\begin{proof}
(1) Let us assume that $f:\mathfrak{g}\stackrel{\sim}{\rightarrow}\mathfrak{h}$ is a quasi-isomorphism of filtered $L_{\infty}$-algebras
inducing a quasi-isomorphism at each stage of the filtration $F_rf:F_r\mathfrak{g}\stackrel{\sim}{\rightarrow}F_r\mathfrak{h}$.
Since we are working over a field, there are cofiber sequences $F_r\mathfrak{g}\hookrightarrow \mathfrak{g}\rightarrow \mathfrak{g}/F_r\mathfrak{g}$
and $F_r\mathfrak{h}\hookrightarrow \mathfrak{h}\rightarrow \mathfrak{h}/F_r\mathfrak{h}$ for every $r$. The map $f$ induces a morphism of cofiber sequences,
hence a morphism of long exact sequences of homology groups. One concludes by the Five Lemma that $f$
induces a quasi-isomorphism $\mathfrak{g}/F_r\mathfrak{g}\stackrel{\sim}{\rightarrow}\mathfrak{h}/F_r\mathfrak{h}$ for every $r$.
These quasi-isomorphisms gather into a morphism of towers of surjections
\[
\xymatrix{
\hat{\mathfrak{g}}\ar@{->>}[r]\ar[d] & ...\ar@{->>}[r] & \mathfrak{g}/F_{r+1}\mathfrak{g}\ar[d]^{\sim}\ar@{->>}[r] & \mathfrak{g}/F_r\mathfrak{g}
\ar@{->>}[r]\ar[d]^{\sim} & ...\ar@{->>}[r] & 0 \\
\hat{\mathfrak{h}}\ar@{->>}[r] & ...\ar@{->>}[r] & \mathfrak{h}/F_{r+1}\mathfrak{h}\ar@{->>}[r] & \mathfrak{h}/F_r\mathfrak{h}
\ar@{->>}[r] & ...\ar@{->>}[r] & 0
}
\]
which is a weak equivalence of towers of fibrations of fibrant objects in the projective model structure of
$Ch_{\mathbb{K}}$, thus inducing a quasi-isomorphism between the limits $\hat{f}:\hat{\mathfrak{g}}\stackrel{\sim}{\rightarrow}
\hat{\mathfrak{h}}$ by general properties of diagrams in model categories. For this we use the dual statement of Proposition 2.5 of \cite{CS1},
with fibrations of fibrant objects and limits instead of cofibrations of cofibrant objects and colimits.

(2) Let us assume that $f:\mathfrak{g}\twoheadrightarrow \mathfrak{h}$ is a surjection inducing a surjection at each stage of the filtration
$F_rf:F_r\mathfrak{g}\twoheadrightarrow F_r\mathfrak{h}$. For every $r$, the map $f$ obviously induces a surjection
$\overline{f}_r:\mathfrak{g}/F_r\mathfrak{g}\twoheadrightarrow\mathfrak{h}/F_r\mathfrak{h}$, hence a termwise fibration of towers of fibrations between fibrant objects
\[
\xymatrix{
\hat{\mathfrak{g}}\ar@{->>}[r]\ar[d] & ...\ar@{->>}[r] & \mathfrak{g}/F_{r+1}\mathfrak{g}\ar@{->>}[d]\ar@{->>}[r] & \mathfrak{g}/F_r\mathfrak{g}
\ar@{->>}[r]\ar@{->>}[d] & ...\ar@{->>}[r] & 0 \\
\hat{\mathfrak{h}}\ar@{->>}[r] & ...\ar@{->>}[r] & \mathfrak{h}/F_{r+1}\mathfrak{h}\ar@{->>}[r] & \mathfrak{h}/F_r\mathfrak{h}
\ar@{->>}[r] & ...\ar@{->>}[r] & 0.
}
\]
We would like to apply the dual statement of Proposition 2.6 of \cite{CS1}.
It is actually an explicit formulation of the fact
that the limit functor is a right Quillen functor for the injective model structure on diagrams, which coincide
with the Reedy model structure on diagrams indexed by inverse Reedy categories.
It implies that the induced morphism between the limits is a fibration, that is, the map
$\hat{f}$ is surjective.
For this, we have to prove, for every $r$, the following. Let us fix an integer $r$.
Let us denote by $p_{\mathfrak{g}}:\mathfrak{g}\twoheadrightarrow \mathfrak{g}/F_{r-1}\mathfrak{g}$ and  $p_{\mathfrak{h}}:\mathfrak{h}\twoheadrightarrow \mathfrak{h}/F_{r-1}\mathfrak{h}$ the projections.
The maps $\overline{f}_{r-1}$ and $\overline{p_{\mathfrak{h}}}:\mathfrak{h}/F_r\mathfrak{h}\twoheadrightarrow \mathfrak{h}/F_{r-1}\mathfrak{h}$ (induced by $p_{\mathfrak{h}}$)
give a pullback
\[
\xymatrix{
\mathfrak{h}/F_r\mathfrak{h}\times_{\mathfrak{h}/F_{r-1}\mathfrak{h}}\mathfrak{g}/F_{r-1}\mathfrak{g} \ar[r] \ar[d] & \mathfrak{g}/F_{r-1}\mathfrak{g} \ar@{->>}[d]^-{\overline{f}_{r-1}} \\
\mathfrak{h}/F_r\mathfrak{h} \ar@{->>}[r]_-{\overline{p_{\mathfrak{h}}}} & \mathfrak{h}/F_{r_1}\mathfrak{h} },
\]
and the maps $\overline{f}_r$ and $\overline{p_{\mathfrak{g}}}:\mathfrak{g}/F_r\mathfrak{g}\twoheadrightarrow \mathfrak{g}/F_{r-1}\mathfrak{g}$ induces a map
\[
\mathfrak{g}/F_r\mathfrak{g}\rightarrow \mathfrak{h}/F_r\mathfrak{h}\times_{\mathfrak{h}/F_{r-1}\mathfrak{h}}\mathfrak{g}/F_{r-1}\mathfrak{g}
\]
defined by
\[
\overline{t}\in \mathfrak{g}/F_r\mathfrak{g}\mapsto (\overline{f(t)},\overline{p_{\mathfrak{g}}(t)})
\]
(we use the notation $\overline{(-)}$ for the equivalences classes).
To obtain a fibration between the limits we need to check that this map is surjective, that is, this map is a fibration.
The pullback $\mathfrak{h}/F_r\mathfrak{h}\times_{\mathfrak{h}/F_{r-1}\mathfrak{h}}\mathfrak{g}/F_{r-1}\mathfrak{g}$ consists in pairs $(\overline{x},\overline{y})$
such that $\overline{x}\in \mathfrak{h}/F_r\mathfrak{h}$, $\overline{y}\in \mathfrak{g}/F_{r-1}\mathfrak{g}$ and $\overline{p_{\mathfrak{h}}(x)}=\overline{f(y)}$.
Since $\overline{p_{\mathfrak{g}}}:\mathfrak{g}/F_r\mathfrak{g}\twoheadrightarrow \mathfrak{g}/F_{r-1}\mathfrak{g}$ is surjective, there exists $\overline{y'}\in \mathfrak{g}/F_r\mathfrak{g}$
such that $\overline{y}=\overline{p_{\mathfrak{g}}(y')}$. Thus we get
\[
\overline{p_{\mathfrak{h}}(x)}=\overline{(f\circ p_{\mathfrak{g}})(y')}=\overline{(p_{\mathfrak{h}}\circ f)(y')},
\]
hence $\overline{x}=\overline{f(y')}+\overline{x'}$ for a certain $x'\in F_{r-1}\mathfrak{h}$.
Since $F_{r-1}f:F_{r-1}\mathfrak{g}\rightarrow F_{r-1}\mathfrak{h}$ by hypothesis on $f$, there exists $x''\in F_{r-1}\mathfrak{g}$ such that
$x'=f(x'')$, which implies that
\[
\overline{x}=\overline{f(y'+x'')}.
\]
Now, $\overline{p_{\mathfrak{g}}(y'+x'')}=\overline{y}+\overline{p_{\mathfrak{g}}(x'')}=\overline{y}$ since $x''$ belongs to $F_{r-1}\mathfrak{g}$.
We finally produced an inverse image of $(\overline{x},\overline{y})$ in $\mathfrak{g}/F_r\mathfrak{g}$.
\end{proof}

The following theorem is a generalization of Getzler's invariance results \cite{Get} to $L_{\infty}$ algebras which are only complete,
not nilpotent:
\begin{thm}
(1) Let $f:\mathfrak{g}\twoheadrightarrow\mathfrak{h}$ be a surjection of complete $L_{\infty}$ algebras inducing surjections at each stage of the filtration.
Suppose that for every integer $r$, the $L_{\infty}$ algebras $\mathfrak{g}/F_r\mathfrak{g}$ and $\mathfrak{h}/F_r\mathfrak{h}$ are profinite. Then $MC_{\bullet}(f)$ is a fibration of Kan complexes.

(2) Let $f:\mathfrak{g}\twoheadrightarrow\mathfrak{h}$ be a quasi-isomorphism of complete $L_{\infty}$ algebras inducing quasi-isomorphisms at each stage of the filtration.
Suppose that $\mathfrak{g}$ and $\mathfrak{h}$ are concentrated in non-negative degrees.
Then $MC_{\bullet}(f)$ is a weak equivalence of Kan complexes.
\end{thm}
\begin{proof}
We use the same notations as in the proof of Proposition 3.1.
Let $\mathfrak{g}$ and $\mathfrak{h}$ be two complete $L_{\infty}$-algebras.
Recall from Proposition 2.6 the towers of fibration of Kan complexes
\[
MC_{\bullet}(\mathfrak{g})\rightarrow...\twoheadrightarrow MC_{\bullet}(\mathfrak{g}/F_{r+1}\mathfrak{g})\twoheadrightarrow MC_{\bullet}(\mathfrak{g}/F_r\mathfrak{g})
\twoheadrightarrow...\twoheadrightarrow pt
\]
and
\[
MC_{\bullet}(\mathfrak{h})\rightarrow...\twoheadrightarrow MC_{\bullet}(\mathfrak{h}/F_{r+1}\mathfrak{h})\twoheadrightarrow MC_{\bullet}(\mathfrak{h}/F_r\mathfrak{h})
\twoheadrightarrow...\twoheadrightarrow pt.
\]

(1) Let $f:\mathfrak{g}\twoheadrightarrow\mathfrak{h}$ be a surjection which induces a surjection
$F_rf:F_r\mathfrak{g}\twoheadrightarrow F_r\mathfrak{h}$ at each stage of the filtration. Then it induces a surjection of nilpotent $L_{\infty}$ algebras
$\overline{f_r}:\mathfrak{g}/F_r\mathfrak{g}\twoheadrightarrow\mathfrak{h}/F_r\mathfrak{h}$ for every $r$, hence a
pointwise fibration of tower of fibrations of fibrant objects in simplicial sets
\[
\xymatrix{
... & MC_{\bullet}(\mathfrak{g}/F_{r+1}\mathfrak{g}) \ar@{->>}[r]^-{MC_{\bullet}(\overline{p_{\mathfrak{g}}})} \ar@{->>}[d]_-{MC_{\bullet}(\overline{f_{r+1}})} & MC_{\bullet}(\mathfrak{g}/F_r\mathfrak{g})\ar@{->>}[r] \ar@{->>}[d]^-{MC_{\bullet}(\overline{f_r})} & ... \\
... & MC_{\bullet}(\mathfrak{h}/F_{r+1}\mathfrak{h}) \ar@{->>}[r]_-{MC_{\bullet}(\overline{p_{\mathfrak{h}}})} & MC_{\bullet}(\mathfrak{h}/F_r\mathfrak{h}) \ar@{->>}[r] & ...
}.
\]
We would like to apply the dual version of Proposition 2.6 (3) in \cite{CS1} (that is, we consider a limit of fibrations instead of a colimit of cofibrations).
For this, we have to prove that for every integer $r$, the map
\[
(MC_{\bullet}(\overline{f_{r+1}}),MC_{\bullet}(\overline{p_{\mathfrak{g}}})):
MC_{\bullet}(\mathfrak{g}/F_{r+1}\mathfrak{g})\rightarrow MC_{\bullet}(\mathfrak{h}/F_{r+1}\mathfrak{h})\times_{MC_{\bullet}(\mathfrak{h}/F_r\mathfrak{h})}
MC_{\bullet}(\mathfrak{g}/F_r\mathfrak{g})
\]
is a fibration.
By assumption, the $L_{\infty}$ algebras $\mathfrak{g}:F_r\mathfrak{g}$ and $\mathfrak{h}/F_r\mathfrak{h}$ are profinite, so according to Lemma 2.3 of \cite{Ber} we have natural isomorphisms of simplicial sets
\[
MC_{\bullet}(\mathfrak{h}/F_{r+1}\mathfrak{h}) \cong Mor_{CDGA_{\mathbb{K}}}(C^*(\mathfrak{h}/F_{r+1}\mathfrak{h}),\Omega_{\bullet}),
\]
\[
MC_{\bullet}(\mathfrak{h}/F_r\mathfrak{h}) \cong Mor_{CDGA_{\mathbb{K}}}(C^*(\mathfrak{h}/F_r\mathfrak{h}),\Omega_{\bullet})
\]
and
\[
MC_{\bullet}(\mathfrak{g}/F_r\mathfrak{g}) \cong Mor_{CDGA_{\mathbb{K}}}(C^*(\mathfrak{g}/F_r\mathfrak{g}),\Omega_{\bullet}).
\]
Coproducts in cdgas are defined by the tensor product of chain complexes.
Given that $Mor_{CDGA_{\mathbb{K}}}(-,\Omega_{\bullet})$ transform colimits into limits, we obtain isomorphisms
\begin{eqnarray*}
 & MC_{\bullet}(\mathfrak{h}/F_{r+1}\mathfrak{h})\times_{MC_{\bullet}(\mathfrak{h}/F_r\mathfrak{h})}
 MC_{\bullet}(\mathfrak{g}/F_r\mathfrak{g}) & \\
 \cong & Mor_{CDGA_{\mathbb{K}}}(C^*(\mathfrak{h}/F_{r+1}\mathfrak{h})\otimes_{C^*(\mathfrak{h}/F_r\mathfrak{h})}
 C^*(\mathfrak{g}/F_r\mathfrak{g}),\Omega_{\bullet}) & \\
 \cong & Mor_{CDGA_{\mathbb{K}}}(C^*(\mathfrak{h}/F_{r+1}\mathfrak{h}\times_{\mathfrak{h}/F_r\mathfrak{h}}\mathfrak{g}/F_r\mathfrak{g}),
 \Omega_{\bullet}) & \\
 \cong & MC_{\bullet}(\mathfrak{h}/F_{r+1}\mathfrak{h}\times_{\mathfrak{h}/F_r\mathfrak{h}}\mathfrak{g}/F_r\mathfrak{g})
\end{eqnarray*}
where the second isomorphism holds by Lemma 3.3.
On the other hand, the map
\[
(\overline{f_{r+1}},\overline{p_{\mathfrak{g}}}):\mathfrak{g}/F_{r+1}\mathfrak{g}\twoheadrightarrow \mathfrak{h}/F_{r+1}\mathfrak{h}\times_{\mathfrak{h}/F_r\mathfrak{h}}\mathfrak{g}/F_r\mathfrak{g}
\]
induced by the pullback
\[
\xymatrix{
\mathfrak{g}/F_{r+1}\mathfrak{g}\ar[dr]\ar@/^{-1pc}/[ddr]\ar@/^{1pc}/[drr] & & \\
 & \mathfrak{h}/F_{r+1}\mathfrak{h}\times_{\mathfrak{h}/F_r\mathfrak{h}}\mathfrak{g}/F_r\mathfrak{g} \ar[r]\ar[d] & \mathfrak{g}/F_r\mathfrak{g}\ar[d] \\
 & \mathfrak{h}/F_{r+1}\mathfrak{h}\ar[r] & \mathfrak{h}/F_r\mathfrak{h}
 }
\]
is a surjection of nilpotent $L_{\infty}$ algebras, so we get a fibration of simplicial sets
\[
MC_{\bullet}((\overline{f_{r+1}},\overline{p_{\mathfrak{g}}})):MC_{\bullet}(\mathfrak{g}/F_{r+1}\mathfrak{g})\twoheadrightarrow MC_{\bullet}(\mathfrak{h}/F_{r+1}\mathfrak{h}\times_{\mathfrak{h}/F_r\mathfrak{h}}\mathfrak{g}/F_r\mathfrak{g}).
\]
In the diagram
\[
\xymatrix{
MC_{\bullet}(g/F_{r+1}g)\ar[dr]\ar@/^{-1pc}/[ddr]\ar@/^{1pc}/[drr] & & \\
 & MC_{\bullet}(\mathfrak{h}/F_{r+1}\mathfrak{h})\times_{MC_{\bullet}(\mathfrak{h}/F_r\mathfrak{h})}
 MC_{\bullet}(\mathfrak{g}/F_r\mathfrak{g}) \ar[r]\ar[d] & MC_{\bullet}(\mathfrak{g}/F_r\mathfrak{g})\ar[d] \\
 & MC_{\bullet}(\mathfrak{h}/F_{r+1}\mathfrak{h})\ar[r] & MC_{\bullet}(\mathfrak{h}/F_r\mathfrak{h}),
 }
\]
the factorization $(MC_{\bullet}(\overline{f_{r+1}}),MC_{\bullet}(\overline{p_{\mathfrak{g}}}))$ is unique up to isomorphism by universal property of the pullback,
so this map is isomorphic to the composite of the fibration $MC_{\bullet}((\overline{f_{r+1}},\overline{p_{\mathfrak{g}}}))$ with the isomorphism
$ MC_{\bullet}(\mathfrak{h}/F_{r+1}\mathfrak{h}\times_{\mathfrak{h}/F_r\mathfrak{h}}\mathfrak{g}/F_r\mathfrak{g}) \cong MC_{\bullet}(\mathfrak{h}/F_{r+1}\mathfrak{h})\times_{MC_{\bullet}(\mathfrak{h}/F_r\mathfrak{h})}
MC_{\bullet}(\mathfrak{g}/F_r\mathfrak{g})$
(this composite also makes the diagram commutative by construction).
We conclude that $(MC_{\bullet}(\overline{f_{r+1}}),MC_{\bullet}(\overline{p_{\mathfrak{g}}}))$ is a fibration.

(2) Let $f:\mathfrak{g}\stackrel{\sim}{\rightarrow}\mathfrak{h}$ be a quasi-isomorphism which induces a quasi-isomorphism
$F_rf:F_r\mathfrak{g}\stackrel{\sim}{\rightarrow}F_r\mathfrak{h}$ at each stage of the filtration. Then it induces a quasi-isomorphism
$\mathfrak{g}/F_r\mathfrak{g}\stackrel{\sim}{\rightarrow}\mathfrak{h}/F_r\mathfrak{h}$ for every $r$ (see the proof of Proposition 3.1 (1) ), hence a
weak equivalence of tower of fibrations of fibrant objects in simplicial sets
\[
\xymatrix{
... & MC_{\bullet}(\mathfrak{g}/F_{r+1}\mathfrak{g}) \ar@{->>}[r] \ar[d]_{\sim} & MC_{\bullet}(\mathfrak{g}/F_r\mathfrak{g})\ar@{->>}[r] \ar[d]_{\sim} & ... \\
... & MC_{\bullet}(\mathfrak{h}/F_{r+1}\mathfrak{h}) \ar@{->>}[r] & MC_{\bullet}(\mathfrak{h}/F_r\mathfrak{h}) \ar@{->>}[r] & ...
}.
\]
We conclude just by applying the dual statement of Proposition 2.5 (3) in \cite{CS1}.
\end{proof}
The technical lemma we use in the proof of part (1) is the following:
\begin{lem}
Let $\mathfrak{g}_1\stackrel{f_1}{\rightarrow}\mathfrak{h}\stackrel{f_2}{\leftarrow}\mathfrak{g}_2$ be a zigzag of morphisms of profinite nilpotent $L_{\infty}$ algebras, then the map
\[
C^*(\mathfrak{g}_1)\otimes_{C^*(\mathfrak{h})}C^*(\mathfrak{g}_2)\rightarrow C^*(\mathfrak{g}_1\times_{\mathfrak{h}}\mathfrak{g}_2)
\]
induced by universal property of the pushout
\[
\xymatrix{
C^*(\mathfrak{h})\ar[r]\ar[d] & C^*(\mathfrak{g}_2)\ar[d]\ar@/^{1pc}/[rdd] & \\
C^*(\mathfrak{g}_1)\ar[r]\ar@/^{-1pc}/[drr] & C^*(\mathfrak{g}_1)\otimes_{C^*(\mathfrak{h})}C^*(\mathfrak{g}_2) \ar[dr] & \\
 & & C^*(\mathfrak{g}_1\times_{\mathfrak{h}}\mathfrak{g}_2)
}
\]
is an isomorphism of cdgas.
\end{lem}
\begin{proof}
Recall that limits of $L_{\infty}$ algebras are determined in the underlying category of chain complexes, so the pullback $\mathfrak{g}_1\times_{\mathfrak{h}}\mathfrak{g}_2$ is defined
by a coreflexive equalizer
\[
\xymatrix{
\mathfrak{g}_1\times_{\mathfrak{h}}\mathfrak{g}_2\ar[r] & \mathfrak{g}_1\oplus\mathfrak{g}_2\ar@<1ex>[r]\ar@<-1ex>[r] & \mathfrak{g}_1\oplus\mathfrak{h}\oplus\mathfrak{g}_2 \ar@/^{-1pc}/[l]
}
\]
induced by $f_1$ and $f_2$.
The $L_{\infty}$ algebra structure on the direct sum $\mathfrak{g}_1\oplus\mathfrak{g}_2$ is defined componentwise, thus, by construction of the Chevalley-Eilenberg algebra,
the isomorphism of exterior algebras
\[
\Lambda(\mathfrak{g}_1^*)\otimes\Lambda(\mathfrak{g}_2^*)\cong\Lambda((\mathfrak{g}_1\oplus\mathfrak{g}_2)^*)
\]
induces an isomorphism of Chevalley-Eilenberg algebras
\[
C^*(\mathfrak{g}_1)\otimes C^*(\mathfrak{g}_2)\cong C^*(\mathfrak{g}_1\oplus\mathfrak{g}_2).
\]
The cdga $C^*(\mathfrak{g}_1\times_{\mathfrak{h}}\mathfrak{g}_2)$ is obtained by the reflexive coequalizer
\[
\xymatrix{
C^*(\mathfrak{g}_1\oplus\mathfrak{g}_2)\ar@<1ex>[r]\ar@<-1ex>[r] & C^*(\mathfrak{g}_1\oplus\mathfrak{h}\oplus\mathfrak{g}_2) \ar@/^{-1pc}/[l] \ar[r] & C^*(\mathfrak{g}_1\times_{\mathfrak{h}}\mathfrak{g}_2)
}
\]
induced by the coreflexive equalizer defining $\mathfrak{g}_1\times_{\mathfrak{h}}\mathfrak{g}_2$.
This coequalizer is naturally isomorphic to the coequalizer
\[
coeq(C^*(\mathfrak{g}_1)\otimes C^*(\mathfrak{g}_2)\rightrightarrows C^*(\mathfrak{g}_1)\otimes C^*(\mathfrak{h})\otimes C^*(\mathfrak{g}_2))
\]
which is exactly $C^*(\mathfrak{g}_1)\otimes_{C^*(\mathfrak{h})}C^*(\mathfrak{g}_2)$.
Moreover, this isomorphism is the map induced by universal property of the pushout $C^*(\mathfrak{g}_1)\otimes_{C^*(\mathfrak{h})}C^*(\mathfrak{g}_2)$, which concludes the proof.
\end{proof}

\begin{cor}
The simplicial Maurer-Cartan functor, precomposed with the completion functor, sends filtration-preserving surjections of filtered $L_{\infty}$-algebras into fibrations
of Kan complexes, and filtration-preserving quasi-isomorphisms of filtered $L_{\infty}$-algebras into weak equivalences of Kan complexes.
\end{cor}

Moreover, the simplicial Maurer-Cartan functor has the following homotopical properties:
\begin{cor}
Let $\mathfrak{g}$ be a filtered $L_{\infty}$ algebra.
The functor
\[
MC_{\bullet}(\mathfrak{g}\hat{\otimes}-):CDGA_{\mathbb{K}}\rightarrow sSet
\]
preserves fibrations and weak equivalences of cdgas.
\end{cor}
\begin{proof}
By definition of the model category structure on $CDGA_{\mathbb{K}}$ (see \cite{Hin}), the fibrations and
weak equivalences are determined by the forgetful functor $CDGA_{\mathbb{K}}\rightarrow Ch_{\mathbb{K}}$.
This implies that fibrations and weak equivalences of $CDGA_{\mathbb{K}}$ are respectively surjections and
quasi-isomorphisms.
Since the tensor product of $Ch_{\mathbb{K}}$ preserves quasi-isomorphisms, any weak equivalence of cdgas
$A\stackrel{\sim}{\rightarrow}B$ induces a quasi-isomorphism of $L_{\infty}$ algebras
$\mathfrak{g}\otimes A\stackrel{\sim}{\rightarrow}\mathfrak{g}\otimes B$. Moreover, by definition of the induced filtrations on
$\mathfrak{g}\otimes A$ and $\mathfrak{g}\otimes B$, this quasi-isomorphism restricts to a quasi-isomorphism
$F_r(\mathfrak{g}\otimes A)\stackrel{\sim}{\rightarrow}F_r(\mathfrak{g}\otimes B)$ for every integer $r$, hence a quasi-isomorphism
at the level of completions $\mathfrak{g}\hat{\otimes}A\stackrel{\sim}{\rightarrow}\mathfrak{g}\hat{\otimes}B$ by Proposition 3.1.

The tensor product of $Ch_{\mathbb{K}}$ also preserves surjections, so any surjection of cdgas
$A\twoheadrightarrow B$ induces a surjection of $L_{\infty}$ algebras
$\mathfrak{g}\otimes A\twoheadrightarrow \mathfrak{g}\otimes B$. Moreover, by definition of the induced filtrations on
$\mathfrak{g}\otimes A$ and $\mathfrak{g}\otimes B$, this surjection restricts to a surjection
$F_r(\mathfrak{g}\otimes A)\twoheadrightarrow F_r(\mathfrak{g}\otimes B)$ for every integer $r$, hence a surjection
at the level of completions $\mathfrak{g}\hat{\otimes}A\twoheadrightarrow \mathfrak{g}\hat{\otimes}B$ by Proposition 3.1.

The result follows by Theorem 3.2.
\end{proof}

\subsection{A generalization of Goldman-Millson Theorem}

We also deduce as a corollary a generalization of Goldman-Millson invariance Theorem \cite{GM}:
\begin{thm}
A quasi-isomorphism of filtered $L_{\infty}$-algebras $f:\mathfrak{g}\stackrel{\sim}{\rightarrow}\mathfrak{h}$ restricting to a quasi-isomorphism
at each stage of the filtration induces a bijection of Maurer-Cartan moduli sets
\[
\mathcal{MC}(\hat{\mathfrak{g}})\cong \mathcal{MC}(\hat{\mathfrak{h}}).
\]
\end{thm}
The equivalence between the quotients by gauge groups actions is recovered in the Lie case:
\begin{cor}
A quasi-isomorphism of filtered dg Lie algebras $f:\mathfrak{g}\stackrel{\sim}{\rightarrow}\mathfrak{h}$ restricting to a quasi-isomorphism
at each stage of the filtration induces a bijection of Maurer-Cartan moduli sets
\[
MC(\hat{\mathfrak{g}})/exp(\hat{\mathfrak{g}}^0)\cong MC(\hat{\mathfrak{h}})/exp(\hat{\mathfrak{h}}^0).
\]
\end{cor}

\subsection{Invariance under twisting}

To conclude, we show how this invariance result behaves with respect to twisting by a Maurer-Cartan
element:
\begin{prop}
Let $f:\mathfrak{g}\stackrel{\sim}{\rightarrow}\mathfrak{h}$ be a quasi-isomorphism of complete $L_{\infty}$ algebras
inducing a quasi-isomorphism at each stage of the filtrations.
Let $\varphi$ be a Maurer-Cartan element in $\mathfrak{g}$.
Then $f$ gives a quasi-isomorphism of twisted $L_{\infty}$-algebras $\mathfrak{g}^{\varphi}\stackrel{\sim}{\rightarrow}\mathfrak{h}^{f(\varphi)}$
(see Definition 2.1) inducing a quasi-isomorphism at each stage of the filtrations.
\end{prop}
\begin{proof}
The map $f$ sends any Maurer-Cartan element $\varphi$ of $\mathfrak{g}$ to a Maurer-Cartan element of $\mathfrak{h}$.
Indeed, for every integer  $n$ we have
\[
\sum_{k=1}^n \frac{1}{n!}[f(\varphi)^{\wedge k}] = f(\sum_{k=1}^n \frac{1}{n!}[\varphi^{\wedge k}])
\]
because $f$ is linear and commutes with all the brackets.
Since both terms are converging series (by completeness of $\mathfrak{h}$) we get at the limit
\[
\sum_{k\geq 1} \frac{1}{n!}[f(\varphi)^{\wedge k}] = f(\sum_{k\geq 1} \frac{1}{n!}[\varphi^{\wedge k}])=0.
\]
Moreover, $f$ defines a morphism of $L_{\infty}$ algebras $f:\mathfrak{g}^{\varphi}\rightarrow\mathfrak{h}^{f(\varphi)}$
by the same arguments.
Now, we use a standard spectral sequence argument for filtered complexes. The filtration of $\mathfrak{g}^{\varphi}$, respectively $\mathfrak{h}^{f(\varphi)}$, is exhaustive and complete.
We consider the associated spectral sequences $\{E^r_{p,q}(\mathfrak{g}^{\varphi})\}$ and  $\{E^r_{p,q}(\mathfrak{h}^{\varphi})\}$.
We have a priori no guaranty that such spectral sequences converge: since these filtrations are not bounded
below and the corresponding spectral sequences are not regular, we cannot appply neither the classical convergence
theorem nor the complete convergence theorem. However, with exhaustive and complete filtrations we can use the Eilenberg-Moore comparison theorem (Theorem 5.5.11 in \cite{Wei}). We just have to check that there exists
an integer $r$ such that for every integers $p$ and $q$ the map $f$ induces an isomorphism
\[
f^r:E^r_{p,q}(\mathfrak{g}^{\varphi})\stackrel{\cong}{\rightarrow} E^r_{p,q}(\mathfrak{h}^{f(\varphi)}).
\]
The differential of $\mathfrak{g}^{\varphi}$, noted $d_{\varphi}$, is given by
\begin{eqnarray*}
d_{\varphi}(x) & = & \sum_{k\geq 0}\frac{1}{k!}[\varphi^{\wedge k},x] \\
  & = & d_{\mathfrak{g}}(x) + \sum_{k\geq 1}\frac{1}{k!}[\varphi^{\wedge k},x].
\end{eqnarray*}
A similar formula holds for the differential of $\mathfrak{h}^{f(\varphi)}$.
For any integer $p$, if $x\in F_p\mathfrak{g}$, then $\sum_{k\geq 1}\frac{1}{k!}[\varphi^{\wedge k},x]\in F_{p+1}\mathfrak{g}$.
Thus the restriction of $d_{\varphi}$ to the quotient complex $F_p\mathfrak{g}/F_{p+1}\mathfrak{g}$ is the differential of $\mathfrak{g}$.
Now, since the quasi-isomorphim $f$ induces a quasi-isomorphism at each stage of the filtrations,
it induces a quasi-isomorphism $F_p\mathfrak{g}/F_{p+1}\mathfrak{g}\stackrel{\sim}{\rightarrow}F_p\mathfrak{h}/F_{p+1}\mathfrak{h}$.
This follows from the Five Lemma (see the proof of Proposition 2.7 (1)).
We conclude that $f$ induces a quasi-isomorphism of graded complexes
\[
Gr(f):Gr(\mathfrak{g}^{\varphi})=(\bigoplus_p F_p\mathfrak{g}/F_{p+1}\mathfrak{g},d_{\mathfrak{g}})\stackrel{\sim}{\rightarrow} Gr(\mathfrak{h}^{f(\varphi)}
=(\bigoplus_p F_p\mathfrak{h}/F_{p+1}\mathfrak{h},d_{\mathfrak{h}}),
\]
hence an isomorphism between the $E_1$ pages of the corresponding spectral sequences.
\end{proof}
\begin{cor}
Let $f:\mathfrak{g}\stackrel{\sim}{\rightarrow}\mathfrak{h}$ be a quasi-isomorphism of complete $L_{\infty}$ algebras
inducing a quasi-isomorphism at each stage of the filtrations.
Let $\varphi$ be a Maurer-Cartan element in $\mathfrak{g}$.
Then $f$ induces a weak equivalence of Kan complexes
\[
MC_{\bullet}(\mathfrak{g}^{\varphi})\stackrel{\sim}{\rightarrow} MC_{\bullet}(\mathfrak{h}^{f(\varphi)}).
\]
\end{cor}

\section{The Maurer-Cartan moduli set in algebraic geometry}

We refer the reader to \cite{Gom}, \cite{LM} and \cite{Neu} for the notions of pseudo-functors, prestacks and stacks.
In this section, we compare the notion of global tangent space of the deformation functor associated to a filtered dg Lie algebra with the notion
of tangent complex associated to the quotient stack of the Maurer-Cartan scheme by the action of the gauge group.
We then explain that our result applies to a wide range of deformation complexes of algebraic structures (encoded by Koszul dg properads), and conclude with
a remark about a possible generalization to $L_{\infty}$ deformation complexes.

\subsection{Action groupoids, group schemes and quotient stacks}

In a given category $\mathcal{C}$, there is a general notion of action groupoid associated to the action $\rho:G\times X\rightarrow X$ of a group object $G$
on an object $X$. This is a groupoid object in $\mathcal{C}$ whose "set of objects" is $X$, "set of morphisms" is $X\times G$, the source map $s:X\times G\rightarrow X$
is the projection on the first component and the target map $t=\rho:X\times G\rightarrow X$ is the action of $G$ on $X$. The composition
\[
(X\times G)\times_{s,t}(X\times G)\cong X\times G\times G\rightarrow X\times G
\]
is given by the multiplication of $G$.

In the particular case where $\mathcal{C}=Sch/S$ is the category of schemes over a scheme $S$, for $X$ a noetherian $S$-scheme and $G$ a smooth
affine group $S$-scheme acting on $X$, we obtain a groupoid scheme noted $[X\times G,X]'$ which forms a prestack.
The corresponding pseudo-functor of groupoids
\[
[X\times G,X]':(Sch/S)^{op}\rightarrow Grpd,
\]
where $Grpd$ is the $2$-category of groupoids, is defined by sending any $S$-scheme $U$ to the action groupoid of $G(U)$ on $X(U)$
(we identify the schemes $X$ and $G$ with their functor of points).
The stackification of this prestack gives an algebraic stack $[X\times G,X]$, which is a meaningful candidate for the notion of quotient stack
associated to the action of $G$ on $X$. It turns out that this stack is, indeed, naturally isomorphic to the usual quotient stack $[X/G]$.
The sheaf of groupoids
\[
[X/G]:(Sch/S)^{op}\rightarrow Grpd
\]
defining this algebraic stack sends any $S$-scheme $U$ to a groupoid $[X/G](U)$ whose objects are the zigzags
\[
U\stackrel{\pi}{\leftarrow}E\stackrel{\alpha}{\rightarrow} X
\]
where $\pi$ is a $G$-torsor (or principal $G$-bundle) and $\alpha$ a $G$-equivariant scheme morphism. The morphisms are the $G$-torsor isomorphisms
$f:E\stackrel{\cong}{\rightarrow}E'$ such that $\alpha'\circ f=\alpha$. This example is presented for instance as Examples 2.18, 2.20 and
2.29 in \cite{Gom}, and Example 2.17 in \cite{Neu}.

\subsection{Deformation functors as prestacks}

We denote by $Alg_{\mathbb{K}}$ the category of unitary commutative $\mathbb{K}$-algebras.

\subsubsection{The Maurer-Cartan variety}
Let $\mathfrak{g}$ be a complete dg Lie algebra such that the vector space $\mathfrak{g}^1$ of degree $1$ elements is of finite dimension.
Then, one can choose a basis for $\mathfrak{g}^1$ and the Maurer-Cartan equation then becomes a system of polynomial equations.
The Maurer-Cartan elements thus consists of the locus of these polynomials and form an affine algebraic variety.
Its functor of points is given by
\[
MC(\mathfrak{g}\hat{\otimes}-):Alg_{\mathbb{K}}\rightarrow Set.
\]
The tangent spaces of the Maurer-Cartan variety are easily computed as follows:
\begin{lem}
For every Maurer-Cartan element $\varphi$ of $\mathfrak{g}$, the tangent space of $MC(\mathfrak{g})$ at $\varphi$ is noted $T_{\varphi}MC(\mathfrak{g})$ and given by
\[
T_{\varphi}MC(\mathfrak{g}) = Z^1(\mathfrak{g}^{\varphi})
\]
the vector space of $1$-cocycles of $\mathfrak{g}^{\varphi}$.
\end{lem}
\begin{proof}
The tangent space $T_{\varphi}MC(\mathfrak{g})$ is the set of scheme morphisms
\[
Spec(\mathbb{K}[t]/(t^2))\rightarrow MC(\mathfrak{g})
\]
in the slice category $Spec(\mathbb{K})/Sch$ of schemes under
$Spec(\mathbb{K})$, where the morphism $Spec(\mathbb{K})\rightarrow MC(\mathfrak{g})$ corresponds to the Maurer-Cartan element $\varphi\in MC(\mathfrak{g})$ and the
inclusion of a closed point $Spec(\mathbb{K})\hookrightarrow Spec(\mathbb{K}[t]/(t^2))$ is induced by the algebra augmentation $\mathbb{K}[t]/(t^2)\rightarrow\mathbb{K}$.
Thanks to Corollary 2.5, it is identified with the subset of $MC(\mathfrak{g}\otimes\mathbb{K}[t]/(t^2))$ of Maurer-Cartan elements of the form $\varphi+\varphi_1t$.
An element $\varphi+\varphi_1t$ of $\mathfrak{g}\oplus \mathfrak{g}\otimes\mathbb{K}.t$ is a Maurer-Cartan element if and only if $|\varphi_1|=|\varphi|=1$ and
\[
d(\varphi+\varphi_1t)+\frac{1}{2}[\varphi+\varphi_1t,\varphi+\varphi_1t]=0.
\]
We have
\begin{eqnarray*}
 & d(\varphi+\varphi_1t)+\frac{1}{2}[\varphi+\varphi_1t,\varphi+\varphi_1t] & \\
 = & d\varphi+(d\varphi_1)t+\frac{1}{2}([\varphi,\varphi]+[\varphi,\varphi_1]t
+[\varphi_1,\varphi]t+[\varphi_1,\varphi_1]t^2) & \\
 = & d\varphi + \frac{1}{2}[\varphi,\varphi] + (d\varphi_1)t + \frac{1}{2}([\varphi,\varphi_1]-(-1)^{|\varphi||\varphi_1|}[\varphi,\varphi_1])t & \\
 = & (d\varphi_1)t + [\varphi,\varphi_1]t & \\
 = & (d_{\varphi}\varphi_1)t. & \\
\end{eqnarray*}
The second line of these equalities holds because $t^2=0$ and the Lie bracket is antisymmetric, the third line holds because $\varphi$ satisfies the Maurer-Cartan
equation in $\mathfrak{g}$ and $|\varphi|=|\varphi_1|=1$, and the last line holds by definition of $d_{\varphi}$.
Hence the equation
\[
d(\varphi+\varphi_1t)+\frac{1}{2}[\varphi+\varphi_1t,\varphi+\varphi_1t]=0
\]
is equivalent to
\[
d_{\varphi}\varphi_1=0,
\]
which means that $\varphi_1$ is a $1$-cocycle of $\mathfrak{g}^{\varphi}$.
\end{proof}

\subsubsection{The deformation functor of a complete dg Lie algebra}
For every $A\in Art_{\mathbb{K}}$, the action of $exp(\mathfrak{g}^0\otimes m_A)$ on $MC(\mathfrak{g}\otimes m_A)$ is functorial in $A$,
so that the quotient gives a well defined functor
\[
MC(\mathfrak{g}\hat{\otimes}-)/exp(\mathfrak{g}^0\otimes -):Art_{\mathbb{K}}\rightarrow Set.
\]
naturally isomorphic to the functor $\pi_0MC_{\bullet}(\mathfrak{g}\otimes-)$ of moduli sets of Maurer-Cartan elements defined before.
This functor of artinian algebras is called the deformation functor of $\mathfrak{g}$ and noted $Def_{\mathfrak{g}}$.

The tangent space $t_{Def_{\mathfrak{g}}}$ of this deformation functor, in the sense of \cite{Sch}, is the evaluation $Def_{\mathfrak{g}}(\mathbb{K}[t]/(t^2))$ of $Def_{\mathfrak{g}}$ on the
algebra of dual numbers. A computation analoguous to the proof of Lemma 4.1 shows that $t_{MC(\mathfrak{g})}=Z^1\mathfrak{g}$, and $t_{exp(\mathfrak{g}^0)}=\mathfrak{g}^0\otimes\mathbb{K}t$ acts on
$t_{MC(\mathfrak{g})}$ via the differential $d:\mathfrak{g}^0\rightarrow Z^1\mathfrak{g}$, hence $t_{Def_{\mathfrak{g}}}=H^1\mathfrak{g}$ (see Section 3 of \cite{Man}).

We recall the following classical result:
\begin{thm}(Cartier \cite{Car})
Every algebraic group over a field of characteristic zero is smooth.
\end{thm}
An algebraic group over a field of characteristic zero is thus a particular case of smooth affine group scheme.
We consider the case of $exp(\mathfrak{g}^0)$, whose associated functor of points is $exp(\mathfrak{g}^0\otimes-)$.
The deformation functor of $\mathfrak{g}$ can be trivially extended to all algebras, and one makes it into a pseudofunctor with values in groupoids
\[
\underline{Def}_{\mathfrak{g}}:Alg_{\mathbb{K}}\rightarrow Grpd
\]
which sends any commutative algebra $A$ to the Deligne groupoid of $\mathfrak{g}\hat{\otimes}m_A$, that is,
the small category whose objects are the Maurer-Cartan elements of $\mathfrak{g}\hat{\otimes}m_A$ and morphisms are
given by the action of the gauge group $exp(\mathfrak{g}^0\otimes m_A)$ (hence they are all invertible).
This is a special case of the notion of action groupoid defined in Section 3.1 for the action of an affine group scheme on a scheme,
hence $\underline{Def}_{\mathfrak{g}}$ is a prestack whose stackification is naturally isomorphic to the quotient stack $[MC(\mathfrak{g})/exp(\mathfrak{g}^0)]$.

The question that naturally arises from such an interpretation is: what is the link between the notion of tangent space of a deformation functor $Def_{\mathfrak{g}}$
in the sense of \cite{Sch} and the notion of tangent complex of the quotient stack $[MC(\mathfrak{g})/exp(\mathfrak{g}^0)]$ in algebraic geometry ?

\subsection{The (truncated) tangent complex of a quotient stack}

Let us denote by $pt$ the affine scheme $Spec(\mathbb{K})$ and $D$ the affine scheme $Spec(\mathbb{K}[t]/(t^2))$ (the infinitesimal disk).
The scheme $D$ is equipped with operations
\[
+:D\rightarrow D\coprod_{pt}D
\]
and, for every $\lambda\in\mathbb{K}$,
\[
\lambda.(-):D\rightarrow D
\]
coming respectively from
\begin{eqnarray*}
\mathbb{K}[t_1,t_2]/(t_1^2,t_2^2,t_1t_2) & \rightarrow & \mathbb{K}[t]/(t^2) \\
a+bt_1+ct_2 & \mapsto & a+(b+c)t
\end{eqnarray*}
and
\begin{eqnarray*}
\mathbb{K}[t]/(t^2) & \rightarrow & \mathbb{K}[t]/(t^2) \\
a+bt & \mapsto & a+\lambda b t.
\end{eqnarray*}
Let X be an algebraic stack. Let us fix, once and for all, a $\mathbb{K}$-point $x:pt\rightarrow X$. Tangent vectors at $x$ are the lifts $D\rightarrow X$ of $x$
and forms a subgroupoid $T_{X,x}=Mor_{pt/St}(D,X)$ of $Mor_{St}(D,X)$. It is the fiber at $x$ of the map $Mor_{St}(D,X)\rightarrow Mor_{St}(pt,X)$ given by precomposition
with the inclusion of the point.
The operations $+$ and $\lambda.(-)$ of $D$ induce functors
\[
+:T_{X,x}\times T_{X,x}\rightarrow T_{X,x}
\]
and
\[
\lambda.(-):T_{X,x}\rightarrow T_{X,x}
\]
so that $T_{X,x}$ forms a Picard category \cite{Del}.
Such a Picard category corresponds to a two-term chain complex $\mathbb{T}_{X,x}:T_{X,x}^{-1}\stackrel{d}{\rightarrow} T_{X,x}^0$ concentrated in degrees $-1$ and $0$.
\begin{defn}
The complex $\mathbb{T}_{X,x}$ is the (truncated) tangent complex of $X$ at $x$.
\end{defn}
The $0^{th}$ cohomology group $H^0\mathbb{T}_{X,x}$ is the group of isomorphism classes of lifts $D\rightarrow X$ of $x$, and $H^{-1}\mathbb{T}_{X,x}$ is the Lie algebra
of the automorphism group of $x$.
\begin{rem}
In the converse direction, given a two-term complex $C:C^{-1}\stackrel{d}{\rightarrow}C^0$, one can associate a Picard category $\mathcal{C}$ defined
by $ob(\mathcal{C})=C^0$ and $Mor_{\mathcal{C}}(x,y)=\{f\in C^{-1}|df=y-x\}$. Then $H^0$ is the group of isomorphism classes of objects of $\mathcal{C}$
and $H^{-1}C$ is the automorphism group of the unit $Aut(1_{\mathcal{C}})$.
\end{rem}
There is actually a categorical equivalence between such kind of complexes and Picard stacks stated more precisely as Proposition 1.4.15 in \cite{Del}.

Now we consider the particular case of a quotient stack $[X/G]$, where $X$ is a noetherian scheme and $G$ a smooth algebraic group acting on $X$.
This algebraic stack admits the following atlas. Consider the object of $[X/G](X)$ defined by the trivial $G$-torsor $G\times X\rightarrow X$
and the group action $G\times X\rightarrow X$. Via the natural isomorphism $[X/G](X)\cong Mor_{St}(X,[X/G])$, it is sent to a surjective map
$u:X\twoheadrightarrow [X/G]$ which forms the desired atlas (see Example 2.29 in \cite{Gom}).
Two lifts $x',x'':pt\rightarrow X$ of a given $\mathbb{K}$-point $x:pt\rightarrow [X/G]$ along this atlas lie in the same orbit under the action of $G$, that is,
the $\mathbb{K}$-points of $[X/G]$ represent the orbits of the action of $G$ and the automorphism group of each such point $x$ is the stabilizer group $G_x$.
Let $x$ be a point of $[X/G]$ and $x'\in X(\mathbb{K})$ be a lifting of $x$.
Let $\phi:G\rightarrow X$ be the scheme morphism defined by sending each $g\in G$ to $x'.g\in X$. Since $\phi$ sends the neutral element
$e$ of $G$ to $x'$, its differential gives a map of tangent spaces
\[
d\phi:T_eG\rightarrow T_{x'}X
\]
which is exactly the tangent complex we are searching for. Since the automorphism group of $x$ is $G_x$, we have $H^{-1}\mathbb{T}_{[X/G],x}=Lie(G_x)$.
Isomorphism classes of infinitesimal deformations of $x$ are given by $H^0\mathbb{T}_{[X/G],x}=T_xX/Im(d\phi)$, which is what one should expect:
two deformations are equivalent if and only if they are related by the action of $G$.

Let $\mathfrak{g}$ be a complete dg Lie algebra such that $\mathfrak{g}^1$ is of finite dimension. Then $MC(\mathfrak{g})$ is an affine algebraic variety with an action of the gauge group $exp(\mathfrak{g}^0)$.
Recall that $\mathfrak{g}^0$ is the Lie algebra of $exp(\mathfrak{g}^0)$ and $T_{\varphi}MC(\mathfrak{g})=Z^1(\mathfrak{g}^{\varphi})$ according to Lemma 4.1.
A point of the quotient stack $[MC(\mathfrak{g})/exp(\mathfrak{g}^0)]$ then represents an equivalence class of Maurer-Cartan elements.
\begin{lem}
The tangent complex of $[MC(\mathfrak{g})/exp(\mathfrak{g}^0)]$ at a point with chosen representative $\varphi$
is the two components chain complex $-d_{\varphi}:\mathfrak{g}^0\rightarrow Z^1(\mathfrak{g}^{\varphi})$, where $d_{\varphi}$ is the differential
of $\mathfrak{g}^{\varphi}$.
\end{lem}
\begin{proof}
We consider the orbit map $o_{\varphi}:exp(\mathfrak{g}^0)\rightarrow MC(\mathfrak{g})$ which sends any $exp(\xi)\in exp(\mathfrak{g}^0)$
to $\varphi.exp(\xi)$. This is a morphism of algebraic varieties. The corresponding natural transformation
between the functor of points $exp(\mathfrak{g}^0\hat{\otimes}-)$ and $MC(\mathfrak{g}\hat{\otimes}-)$, also noted $o_{\varphi}$,
is defined by $o_{\varphi}(A)=o_{\varphi\hat{\otimes}1_A}$.

Since $o_{\varphi}:exp(\mathfrak{g}^0)\rightarrow MC(\mathfrak{g})$ sends the neutral element $exp(0)$ to $\varphi$,
it defines a tangent map $d_{exp(0)}o_{\varphi}:T_{exp(0)}exp(\mathfrak{g}^0)\rightarrow T_{\varphi}MC(\mathfrak{g})$.
The tangent spaces are defined by
\[
T_{exp(0)}exp(\mathfrak{g}^0)=Mor_{Spec(\mathbb{K})/Sch}(Spec(\mathbb{K}[t]/(t^2)),exp(\mathfrak{g}^0))
\]
and
\[
T_{\varphi}MC(\mathfrak{g})= Mor_{Spec(\mathbb{K})/Sch}(Spec(\mathbb{K}[t]/(t^2)),MC(\mathfrak{g})),
\]
where $Spec(\mathbb{K})/Sch$ is the category of schemes under $Spec(\mathbb{K})$, and the $\mathbb{K}$-points
$Spec(\mathbb{K})\rightarrow exp(\mathfrak{g}^0)$ and $Spec(\mathbb{K})\rightarrow MC(\mathfrak{g})$ correspond respectively
to $exp(0)$ and $\varphi$ via the Yoneda Lemma.
The map $d_{exp(0)}o_{\varphi}$ is then given by composing a morphism
$Spec(\mathbb{K}[t]/(t^2))\rightarrow exp(\mathfrak{g}^0)$ under $Spec(\mathbb{K})$ with $o_{\varphi}$.

Recall that according to the Yoneda Lemma, there is a bijection
\[
Nat(Hom_{Alg_{\mathbb{K}}}(\mathbb{K}[t]/(t^2),-),MC(\mathfrak{g}\hat{\otimes}-)\stackrel{\cong}{\rightarrow}
MC(\mathfrak{g}\otimes\mathbb{K}[t]/(t^2))
\]
sending a natural transformation $\tau$ to $\tau(\mathbb{K}[t]/(t^2))(id)$, whose inverse sends
an element $x\in MC(\mathfrak{g}\otimes\mathbb{K}[t]/(t^2))$ to the natural transformation $\tau_x$ defined
by $\tau_x(f)=MC(\mathfrak{g}\hat{\otimes}f)(x)$.
The tangent map $d_{exp(0)}o_{\varphi}$ fits in a commutative square
\[
\xymatrix{
T_{exp(0)}exp(\mathfrak{g}^0) \ar[r]^-{d_{exp(0)}o_{\varphi}} & T_{\varphi}MC(\mathfrak{g})\ar[d]^{\cong} \\
\{exp(0)+exp(\xi)t,\xi\in \mathfrak{g}^0\}\ar[r]  \ar[u]^-{\cong}& \{\varphi+\varphi_1t\in MC(\mathfrak{g}\otimes\mathbb{K}[t]/(t^2))\}.
}
\]
The bijection
\[
\{exp(0)+exp(\xi)t,\xi\in \mathfrak{g}^0\}\rightarrow Mor_{Spec(\mathbb{K})/Sch}(Spec(\mathbb{K}[t]/(t^2)),exp(\mathfrak{g}^0))
\]
sends every $exp(0)+exp(\xi)t$ to a natural transformation defined by
\[
f\in Spec(\mathbb{K}[t]/(t^2))(A)\mapsto exp(\mathfrak{g}^0\hat{\otimes}f)(exp(0)+exp(\xi)t)
\]
for every $\mathbb{K}$-algebra $A$.
The tangent map then sends it to the natural transformation defined by
\[
f\in Spec(\mathbb{K}[t]/(t^2))(A)\mapsto o_{\varphi}(A)(exp(\mathfrak{g}^0\hat{\otimes}f)(exp(0)+exp(\xi)t)).
\]
Finally, the bijection
\[
Mor_{Spec(\mathbb{K})/Sch}(Spec(\mathbb{K}[t]/(t^2)),MC(\mathfrak{g}))\mapsto \{\varphi+\varphi_1t\in MC(\mathfrak{g}\otimes\mathbb{K}[t]/(t^2))\}
\]
sends this natural transformation to $o_{\varphi}(\mathbb{K}[t]/(t^2))(exp(\mathfrak{g}^0\hat{\otimes}id)(exp(0)+exp(\xi)t))$,
which is nothing but $o_{\varphi}(\mathbb{K}[t]/(t^2))(exp(0)+exp(\xi)t)$.
The map
\[
d_{exp(0)}o_{\varphi}:\{exp(0)+exp(\xi)t,\xi\in \mathfrak{g}^0\}\rightarrow\{\varphi+\varphi_1t\in MC(\mathfrak{g}\otimes\mathbb{K}[t]/(t^2))\}
\]
is consequently the restriction of $o(\mathbb{K}[t]/(t^2)):exp(\mathfrak{g}^0\otimes\mathbb{K}[t]/(t^2))
\rightarrow MC(\mathfrak{g}\otimes \mathbb{K}[t]/(t^2))$ to $\{exp(0)+exp(\xi)t,\xi\in \mathfrak{g}^0\}$.

Now, by definition $o(\mathbb{K}[t]/(t^2))=o_{\varphi\otimes 1_{\mathbb{K}[t]/(t^2)}}=o_{\varphi\oplus 0.t}$.
Moreover, for any $exp(0)+exp(\xi)t\in exp(\mathfrak{g}^0\otimes\mathbb{K}[t]/(t^2))$ and any
$\varphi+\varphi_1t\in MC(\mathfrak{g}\otimes\mathbb{K}[t]/(t^2))$, the action of $exp(0)+exp(\xi)t$ on $\varphi+\varphi_1t$
is defined by
\[
(\varphi+\varphi_1t).(exp(0)+exp(\xi)t) = e^{[\xi t,-]}(\varphi+\varphi_1t)-\frac{e^{[\xi t,-]}_Id}{[\xi t,-]}(d_{\mathfrak{g}}\xi t).
\]
We have
\begin{eqnarray*}
[\xi t,\varphi+\varphi_1t] & = & [\xi,\varphi]t + [\xi,\varphi_1]t^2 \\
 & = & [\xi,\varphi]t
\end{eqnarray*}
so for any integer $k\geq 1$ the iterate composites $[\xi t,-]^{\circ k}(\varphi+\varphi_1 t)$ are zero.
This implies that
\[
e^{[\xi t,-]}(\varphi+\varphi_1t) = \varphi+\varphi_1t+[\xi,\varphi]t.
\]
Moreover, we have
\begin{eqnarray*}
\frac{e^{[\xi t,-]}_Id}{[\xi t,-]}(d_{\mathfrak{g}}\xi t) & = & \sum_{k=1}^{+\infty}\frac{1}{k!}[\xi t,-]^{\circ (k-1)}(d_{\mathfrak{g}}\xi t) \\
 & = & d_{\mathfrak{g}}\xi t
\end{eqnarray*}
because $[\xi t,-]^{\circ (k-1)}(d_{\mathfrak{g}}\xi t)=0$ for $k\geq 2$.
Consequently, the formula defining the action of $(exp(0)+exp(\xi)t)$ on $(\varphi+\varphi_1t)$ becomes
\begin{eqnarray*}
(\varphi+\varphi_1t).(exp(0)+exp(\xi)t) & = & \varphi+\varphi_1t+[\xi,\varphi]t-d_{\mathfrak{g}}\xi t \\
& = & \varphi + (\varphi_1-d_{\varphi}\xi)t,
\end{eqnarray*}
hence
\[
o_{\varphi}(\mathbb{K}[t]/(t^2))(exp(0)+exp(\xi)t)=\varphi - d_{\varphi}\xi t.
\]
The map $o_{\varphi}(\mathbb{K}[t]/(t^2))$, in turn, fits in a commutative square
\[
\xymatrix{
\{exp(0)+exp(\xi)t,\xi\in \mathfrak{g}^0\}\ar[r]^-{o_{\varphi}(\mathbb{K}[t]/(t^2))}\ar[d]_-{\cong} & \{\varphi+\varphi_1t\in MC(\mathfrak{g}\otimes\mathbb{K}[t]/(t^2))\}\ar[d]^-{\cong} \\
\mathfrak{g}^0 \ar[r] & Z^1\mathfrak{g}^{\varphi}
}
\]
where the left vertical bijection sends $exp(0)+exp(\xi)t$ to $\xi$ and the right vertical bijection
sends $\varphi+\varphi_1 t$ to $\varphi_1$.
Since $o_{\varphi}(\mathbb{K}[t]/(t^2))$ sends $exp(0)+exp(\xi)t$ to $\varphi - d_{\varphi}\xi t$,
we finally obtain
\[
d_{exp(0)}o_{\varphi} = -d_{\varphi}:\mathfrak{g}^0\rightarrow Z^1\mathfrak{g}^{\varphi}.
\]
\end{proof}
We thus answered our question:
\begin{thm}
Let $\mathfrak{g}$ be a complete dg Lie algebra such that $\mathfrak{g}^1$ is of finite dimension.
Let $\varphi$ be a Maurer-Cartan element of $\mathfrak{g}$ and $[\varphi]\in\mathcal{MC}(\mathfrak{g})$ its equivalence class.
The $0^{th}$ cohomology group of the tangent complex of the algebraic stack $[MC(\mathfrak{g})/exp(\mathfrak{g}^0)]$ at $[\varphi]$ is the global tangent space of the deformation
functor $Def_{\mathfrak{g}^{\varphi}}$, that is,
\[
H^0\mathbb{T}_{[MC(\mathfrak{g})/exp(\mathfrak{g}^0)],[\varphi]} = t_{Def_{\mathfrak{g}^{\varphi}}} = H^1\mathfrak{g}^{\varphi}.
\]
\end{thm}

\subsection{Removing the assumption about $dim(\mathfrak{g}^1)$}

The assumption about the dimension of $\mathfrak{g}^1$ is restrictive. It is in general not satisfied in algebraic deformation theory, see for instance
the definition of the deformation complex of a differential graded algebra over an operad (Chapter 12 of \cite{LV}).
However, for any complete dg Lie algebra $\mathfrak{g}$ we have
\[
MC(\mathfrak{g})=lim MC(\mathfrak{g}/F_r\mathfrak{g})
\]
and more generally
\[
MC(\mathfrak{g}\otimes A)=lim MC(\mathfrak{g}/F_r\mathfrak{g}\otimes A)
\]
for any commutative algebra $A$.
If the vector space of degree $1$ elements of $\mathfrak{g}/F_r\mathfrak{g}$ is of finite dimension for every integer $r$, then $MC(\mathfrak{g})$ is a limit of algebraic varieties, so it is a priori
not an algebraic variety anymore but still a scheme. Consequently, all the preceding result still holds.

\subsection{Deformation complexes of algebraic structures}

The assumption above is enough to recover all deformation complexes coded by Koszul operads and Koszul properads.
To fully define all the concepts we need to justify this claim, we would have to go further in the theory of operads and props at a point which is out of
the scope of this paper. Therefore, we just provide some of the ideas behind, explain the
main argument and refer the reader to \cite{LV}, \cite{Mar}, \cite{MV} and \cite{Val}.
for more details about the underlying concepts and objects.

Let us consider a collection of chain complexes $\{P(m,n)\}$ representing "formal operations" with $m$
inputs and $n$ outputs. Each $P(m,n)$ is equipped with a right action of the symmetric group $\Sigma_m$
and a left action of $\Sigma_n$ commuting with each other. These actions represent the permutations of
the inputs and the outputs of the formal operations. Such a collection is called a $\Sigma$-biobject.
A properad is a monoid in the category of $\Sigma$-biobjects for a certain product, the connected composition
product. Being a monoid for this connected composition product means that one can compose the operations of $P$
along $2$-levelled directed connected graphs (with no loops), whose vertices are indexed by elements of $P$.
For a detailed definition we refer the reader to \cite{Val}, where properads where introduced.

Properads are used to parametrize various kinds of algebras and bialgebras (thus generalizing operads, which
encode only algebras). To do this, let us consider a chain complex $X$. We associate to $X$ a properad $End_X$
called the endomorphism properad of $X$ and defined by $End_X(m,n)=Hom_{\mathbb{K}}(X^{\otimes m},
X^{\otimes n})$, where $Hom_{\mathbb{K}}$ is the differential graded hom of chain complexes.
A $P$-algebra structure on $X$ is then a properad morphism
\[
P\rightarrow End_X,
\]
that is, a collection of chain morphisms
\[
P(m,n)\rightarrow Hom_{\mathbb{K}}(X^{\otimes m},X^{\otimes n})
\]
equivariant for the actions of the symmetric groups and compatible with the composition products.
This means that each "formal operation" of $P$ is sent to a map $X^{\otimes m}\rightarrow X^{\otimes n}$,
and the way these operations compose in $P$ gives the relations satisfied by such maps.
For instance, if $P=Ass$ is the properad of associative algebras, then $P(2,1)$ is generated
by an element $\mu$ which is sent to an associative product $\mu_X:X\otimes X\rightarrow X$ on $X$.
Properads form a model category (see Appendix of \cite{MV}) in which one can generalize
Koszul duality theory to obtain "small" cofibrant resolutions for a certain kind of properads
(Koszul properads, see \cite{Val}).

Let $P$ be a dg properad and $X$ be a chain complex.
Let $P_{\infty}=(\mathcal{F}(s^{-1}\overline{C}),\partial)$ be a cofibrant resolution of $P$ given by a quasi-free dg properad, generated by a $\Sigma$-biobject $C$.
We suppose that each $C(m,n)$ is a chain complex of finite dimension in each degree, and that for a fixed degree $d$, $C_d(m,n)$ is non zero only for a finite number of couples
$(m,n)$. By Theorem 5 of \cite{MV}, the complex of $\Sigma$-biobjects homomorphisms $\mathfrak{g}_{P,X}=Hom_{\Sigma}(\overline{C},End_X)$ is a complete $L_{\infty}$ algebra.
Moreover, it turns out that it satisfies the assumption above. To see this, we use the idea of Proposition 15 in \cite{Mar}
that $\mathfrak{g}$ admits a certain decomposition $\mathfrak{g}=\prod_s\mathfrak{g}_s$ such that the filtration defined by
$F_r\mathfrak{g}=\prod_{s\geq l}\mathfrak{g}_s$ makes $\mathfrak{g}$ into a complete $L_{\infty}$ algebra. We then have $\mathfrak{g}/F_r\mathfrak{g}=\prod_{s=1}^{l-1}\mathfrak{g}_s=\bigoplus_{s=1}^{l-1}\mathfrak{g}_s$, in particular
$(\mathfrak{g}/F_r\mathfrak{g})^1=\bigoplus_{s=1}^{l-1}\mathfrak{g}_s^1$. The assumptions on $C$ directly implies that each $\mathfrak{g}_s^1$ is of finite dimension, hence $(\mathfrak{g}/F_r\mathfrak{g})^1$ is of finite dimension for every
integer $r$.
To obtain a complete Lie algebra instead of an $L_{\infty}$-algebra, one has to add the assumption that $\partial(s^{-1}\overline{C})\subset \mathcal{F}(s^{-1}\overline{C})^{(\leq 2)}$, that is, the image of the generators under the differential consists in graphs with at most two vertices (see Theorem 5 in \cite{MV}).

Let us also note that Maurer-Cartan elements of $\mathfrak{g}_{P,X}$ are exactly the $P_{\infty}$-algebra structures on X.
If $X$ is a $P$-algebra, thus equipped with a map $P\rightarrow End_X$, then the deformation complex of $X$ is defined by the derivations complex $Der(P_{\infty},End_X)$.
These derivations are defined with respect to the properad morphism $\varphi:P_{\infty}\rightarrow P\rightarrow End_X$, which is actually a Maurer-Cartan element
of $\mathfrak{g}_{P,X}$, and we have an isomorphism $\mathfrak{g}_{P,X}^{\varphi}\cong Der(P_{\infty},End_X)$.

\begin{thm}
Let $P$ be a Koszul dg properad. Let $P_{\infty}=(\mathcal{F}(s^{-1}\overline{C}),\partial)$ be the associated cofibrant resolution via Koszul duality.
Suppose that each $C(m,n)$ is a chain complex of finite dimension in each degree, and that for a fixed degree $d$, $C_d(m,n)$ is non zero only for a finite number of couples
$(m,n)$. Then:

(1) The Maurer-Cartan moduli set of $\mathfrak{g}_{P,X}$ forms a prestack, whose stackification is naturally isomorphic to the quotient stack $[MC(\mathfrak{g}_{P,X})/exp(\mathfrak{g}_{P,X}^0)]$;

(2)Given an equivalence class $[\varphi]\in\mathcal{MC}(\mathfrak{g}_{P,X})$, the $0^{th}$ cohomology group of the tangent complex of $[MC(\mathfrak{g}_{P,X})/exp(\mathfrak{g}_{P,X}^0)]$ is
$H^1\mathfrak{g}_{P,X}^{\varphi}$, that is, the tangent space of the deformation functor $Def_{\mathfrak{g}_{P,X}^{\varphi}}$.
\end{thm}

\begin{rem}
For complete $L_{\infty}$ algebras, the geometric interpretation is more involved, since the Maurer-Cartan moduli set cannot be defined by
a quotient under the action of a gauge group. However, there is a way to prove that the associated functor is naturally
isomorphic to the Maurer-Cartan moduli set functor of a complete dg Lie algebra.
The main ingredient is the strictification of homotopy algebras over an operad. More precisely, since the category of
dg operads form a model category \cite{Hin}, any operad $P$ admits a cofibrant resolution $\pi:P_{\infty}
\stackrel{\sim}{\rightarrow}P$ which gives rise to a functor
\[
\pi^*:P-Alg\rightarrow P_{\infty}-Alg
\]
from $P$-algebras to $P_{\infty}$-algebras.
This functor is simply defined by precomposing any $P$-algebra structure $P\rightarrow End_X$ with $\pi$.
It turns out that $\pi^*$ admits a left adjoint, hence fitting in an adjunction
\[
\pi_{!}:P-Alg\rightarrow P_{\infty}-Alg:\pi^*.
\]
Algebras over an operad in $Ch_{\mathbb{K}}$ form a model category whose weak equivalences and fibrations
are the quasi-isomorphisms and surjections of chain complexes, so the right adjoint $\pi^*$ preserves
them by definition and thus form a right Quillen functor. Since $\pi$ is a weak equivalence, one can prove
that the Quillen pair $(\pi_{!},\pi^*)$ is actually a Quillen equivalence.
In particular, the unit of the adjunction is a natural weak equivalence of $P_{\infty}$-algebras
\[
X\stackrel{\sim}{\rightarrow}\pi^*\pi_{!}X,
\]
where $\pi^*\pi_{!}X$ is actually a $P$-algebra equipped with the trivial $P_{\infty}$ structure defined by $P_{\infty}\stackrel{\sim}{\rightarrow}P\rightarrow End_{\pi^*\pi_{!}X}$.
In this way, any $P_{\infty}$-algebra is naturally quasi-isomorphic to a $P$-algebra.
Applying this to the operad $Lie$ encoding dg Lie algebras, we obtain that any $L_{\infty}$ algebra $\mathfrak{g}$
is naturally quasi-isomorphic to the dg Lie algebra $\pi^*\pi_{!}\mathfrak{g}$ via the unit $\eta(\mathfrak{g}):\mathfrak{g}\stackrel{\sim}{\rightarrow}\pi^*\pi_{!}\mathfrak{g}$,
where $\pi:Lie_{\infty}\stackrel{\sim}{\rightarrow} Lie$.

Now let $\mathfrak{g}$ be a filtered $L_{\infty}$ algebra, we have a quasi-isomorphism of $L_{\infty}$ algebras
$\hat{\mathfrak{g}}\stackrel{\sim}{\rightarrow} \pi^*\pi_{!}\hat{\mathfrak{g}}$.
One has to check that this is a quasi-isomorphism of complete $L_{\infty}$ algebras inducing a quasi-isomorphism at each stage of the filtration,
hence a natural weak equivalence of simplicial Maurer-Cartan functors
\[
MC_{\bullet}(\mathfrak{g}\hat{\otimes}-)\stackrel{\sim}{\rightarrow}MC_{\bullet}(\pi^*\pi_{!}\mathfrak{g}\hat{\otimes}-)
\]
inducing, in turn, a natural isomorphism of Maurer-Cartan moduli set functors
\[
\pi_0MC_{\bullet}(\mathfrak{g}\hat{\otimes}-)\stackrel{\cong}{\rightarrow}
\pi_0MC_{\bullet}(\pi^*\pi_{!}\mathfrak{g}\hat{\otimes}-).
\]
We get a natural isomorphism
\[
\mathcal{MC}(\mathfrak{g}\hat{\otimes}-)\stackrel{\cong}{\rightarrow}
MC(\pi^*\pi_{!}\mathfrak{g}\hat{\otimes}-)/exp((\pi^*\pi_{!}\mathfrak{g})^0\otimes -).
\]
Then one has to check that $MC(\pi^*\pi_{!}\hat{\mathfrak{g}})$ is still a scheme to obtain the algebraic stack interpretation of the Maurer-Cartan moduli set
$MC(\pi^*\pi_{!}\mathfrak{g}\hat{\otimes}-)/exp((\pi^*\pi_{!}\mathfrak{g})^0\otimes -)$. Finally, given a Maurer-Cartan element $\varphi$ of $\mathfrak{g}$, Proposition 3.8 implies that
$H^1\mathfrak{g}^{\varphi}$, which is the global tangent space of the deformation functor $Def_{\mathfrak{g}^{\varphi}}$, is the $0^{th}$ cohomology group of the tangent complex
of the quotient stack $[MC(\pi^*\pi_{!}\mathfrak{g})/exp(\pi^*\pi_{!}\mathfrak{g}^0)]$ at the point $\eta(\mathfrak{g})(\varphi)$.
\end{rem}

\end{document}